\documentclass[a4paper, 11pt, reqno]{amsart}

\usepackage[top=1in, bottom=1.25in, left=1.25in, right=1.25in]{geometry}

\allowdisplaybreaks

\numberwithin{equation}{section}

\usepackage[latin1]{inputenc}
\usepackage[T1]{fontenc}
\usepackage[english]{babel}
\usepackage{latexsym,amsmath,amssymb,amsfonts}
\usepackage{dsfont}
\usepackage{color}
\usepackage{graphicx}
\usepackage{float}
\usepackage{esint}
\usepackage{enumerate}

\theoremstyle{plain}
\begingroup
\newtheorem{theorem}{Theorem}[section]
\newtheorem{lemma}[theorem]{Lemma}
\newtheorem{proposition}[theorem]{Proposition}
\newtheorem{corollary}[theorem]{Corollary}
\endgroup
\theoremstyle{definition}
\begingroup
\newtheorem{definition}[theorem]{Definition}
\newtheorem{remark}[theorem]{Remark}

\endgroup
\theoremstyle{remark}
\newcommand{\N}{\mathbb N}
\newcommand{\Z}{\mathbb Z}
\newcommand{\Q}{\mathbb Q}
\newcommand{\R}{\mathbb R}

\renewcommand{\S}{\mathbb S}

\newcommand{\supp}{{\rm supp\,}}

\renewcommand{\l}{{\mathcal L}}

\newcommand{\hno}{{\mathcal H}^{N-1}}

\newcommand{\dd}{\mathrm{d}}

\newcommand{\dhno}{\;\dd{\mathcal H}^{N-1}}

\newcommand{\dx}{\;\dd x}
\newcommand{\dy}{\;\dd y}

\newcommand{\e}{\varepsilon}
\newcommand{\del}{\delta}

\renewcommand{\o}{\Omega}
\newcommand{\rga}{\rightarrow}
\renewcommand{\sp}{H^1(\o;\R^d)}
\newcommand{\cR}{\mathcal R}
\newcommand{\cA}{\mathcal A}
\newcommand{\uCR}{u^C_{\e_n^{\cR}}}

\newcommand{\f}{\mathcal{F}}

\newcommand{\W}{\mathcal W}

\newcommand{\restr}{%
  \,\raisebox{-.127ex}{\reflectbox{\rotatebox[origin=br]{-90}{$\lnot$}}}\,%
}

\newcommand{\average}{{\mathchoice {\kern1ex\vcenter{\hrule height.4pt
width 6pt
depth0pt} \kern-9.7pt} {\kern1ex\vcenter{\hrule height.4pt width 4.3pt
depth0pt}
\kern-7pt} {} {} }}

\title[Homogenization and phase transitions]{A homogenization result in the gradient theory of phase transitions}
\author[Cristoferi, Fonseca,  Hagerty, Popovici]{Riccardo Cristoferi, Irene Fonseca, Adrian Hagerty, Cristina Popovici}

\begin{document}

\begin{abstract}
A variational model in the context of the gradient theory for fluid-fluid phase transitions with small scale heterogeneities is studied.
In particular, the case where the scale $\e$ of the small homogeneities is of the same order of the scale governing the phase transition is considered.
The interaction between homogenization and the phase transitions process will lead, in the limit as $\e\to0$, to an anisotropic interfacial energy.
\end{abstract}

\maketitle

\section{Introduction}

In order to describe the behavior at equilibrium of a fluid under isothermal conditions confined in a container $\o\subset\R^N$ and having two stable phases (or a mixture of two immiscible and non-interacting fluids with two stable phases), Van der Waals in his pioneering work \cite{vanDW} (then rediscovered by Cahn and Hilliard in \cite{CH}) introduced the following Gibbs free energy per unit volume
\begin{equation}\label{eq:envdw}
E_\e(u):=\int_\o \left[\, W(u) + \e^2|\nabla u|^2 \,\right] \dx\,.
\end{equation}
Here $\e>0$ is a small parameter, $W:\R\to[0,+\infty)$ is a double well potential vanishing at two points, say $+1$ and $-1$ (the simplified prototype being $W(t):=(1-t^2)^2$), and $u:\o\to\R$ represents the phase of the fluid, where $u=+1$ correspond to one stable phase and $u=-1$ to the other one.
According to this gradient theory for first order phase transitions, observed stable configurations minimize the energy $E_\e$ under a mass constraint $\int_\o u=m$, for some fixed $m\in(-|\o|,|\o|)$.

The gradient term present in the energy \eqref{eq:envdw} provides a selection criterion among minimizers of $I:u\mapsto\int_\o W(u)\dx$. If neglected then every field $u$ such that $W(u)\equiv0$ in $\o$ and satisfying the mass constraint is a minimizer of $I$.
The singular perturbation $u\mapsto \e^2|\nabla u|^2$ provides a selection criterion and it competes with the potential term in that it penalizes inhomogeneities of $u$ and acts as a regularization for the problem.
In particular, the parameter $\e>0$ is related to the thickness of the transition layer between the two phases.
It was conjectured by Gurtin (see \cite{Gurtin}) that for $0<\e\ll1$ the minimizer $u_\e$ of the energy $E_\e$ will approximate a piecewise constant function, $u$, taking values in the zero set of the potential $W$, and minimizing the surface area $\hno(J_u)$ of the interface separating the two phases. Here $\hno$ denotes the $(N-1)$-dimensional Hausdorff measure and $J_u$ is the set of jump points of $u$.

Gurtin's conjecture has been proved by Modica in \cite{Modica} (see also the work of Sternberg \cite{Stern}) using $\Gamma$-convergence techniques introduced by De Giorgi and Franzoni in \cite{DGFran}.
In particular, it has been showed that
\[
\lim_{\e\to0}\frac{1}{\e}E_\e(u_\e)=\gamma \hno(J_u)\,,
\]
where the constant $\gamma>0$ plays the role of the surface energy density per unit area required to make a transition from one stable phase to the other, and it is given by
\[
\gamma:=\int_{-1}^1 \sqrt{W(t)} \dd t\,.
\]

Several variants of the Van der Waals-Cahn-Hilliard gradient theory for phase transitions have been studied analytically. Here we recall the extension to the case of $d$ non-interacting immiscible fluids, with a vector-valued density $u:\R^N\to\R^d$. In \cite{fonseca89} Fonseca and Tartar treated the case of two stable phases (\emph{i.e.}, the potential $W:\R^d\to[0,\infty)$ has two zeros), while the general case of several stable phases has been solved by Baldo in \cite{baldo}. In \cite{baldo} and \cite{fonseca89} it has been proved that the limit of a sequence $\{u_\e\}_{\e>0}$, where $u_\e$ is a minimizer of $E_\e$, is a minimal partition of the container $\o$, where each set satisfies a volume constraint and corresponds to a stable phase, \emph{i.e.}, a zero of $W$.

Other generalizations of \eqref{eq:envdw} include the work of Bouchitt\'{e} \cite{Bou}, who studied the case of a fluid where its two stable phases change from point to point, in order to treat the situation where the temperature of the fluid is not constant inside the container, but given a priori. From the mathematical point of view, this corresponds to considering the energy \eqref{eq:envdw} with a potential of the form $W(x,u)$ vanishing on the graphs of two non constant functions $z_1, z_2:\o\to\R^d$.
Fonseca and Popovici in \cite{FonPopo} dealt with the vectorial case of the energy \eqref{eq:envdw} where the term $|\nabla u|$ is substituted with a more general expression of the form $h(x,\nabla u)$, while the full coupled singular perturbed problem in the vectorial case, with the energy density of the form $f(x,u,\e\nabla u)$, has been studied by Barroso and Fonseca in \cite{BarFon}.
The case in which Dirichlet boundary conditions are considered was addressed by Owen, Rubinsten and Sternberg in \cite{OwRubSter}, while in \cite{Modicacont} Modica studied the case of a boundary contact energy.
We refer to the works \cite{Stern} of Sternberg and \cite{Ambr} of Ambrosio for the case where the zeros of the potential $W$ are generic compact sets.
Finally, in \cite{KohnStern} Kohn and Sternberg studied the convergence of local minimizers for singular perturbation problems.\\

This paper is part of an ongoing project aimed at studying the interaction between phase transitions and homogenization, namely when small scale heterogeneities are present in the fluids.
In particular, we treat the case of
a mixture of $d$ non-interacting immiscible fluids with two minimizing phases in isothermal conditions.
To be precise, for $\e>0$ we consider the energy
\begin{equation}\label{eq:energynotrescaled}
\mathcal{E}_\e(u):=\int_\o \left[ \, W\left(\frac{x}{\e},u(x)\right) + \e^2 |\nabla u(x)|^2  \,\right] \dx\,,
\end{equation}
where $W:\R^N\times\R^d\to[0,\infty)$ is a double well potential that is $1$-periodic in the first variable and with two zeros (see Section \ref{sec:mainresult} for more precise details on the hypotheses on $W$).
The small scale heterogeneities are modeled by the fast oscillations in the first variable of the potential $W$.

Since $\lim_{\e\to0}\min\mathcal{E}_\e=0$, in order to understand the behavior of minimizing sequences as $\e\to0$ we need to consider the rescaled energy $\f_\e:=\e^{-1}\mathcal{E}_\e$.
In the main result of this paper (see Theorem \ref{thm:main}) we identify the variational limit (in the sense of $\Gamma$-convergence) of the rescaled energies $\f_\e$ as $\e\to0$.
In particular, we will prove that the limiting energy is given by an anisotropic surface functional.
We refer to Section \ref{sec:mainresult} for the precise statement of the result.
Since the scaling $\e^{-1}$ of the energy coincides with the scaling of the fine oscillations in the potential, we expect to observe, in the limit, an interaction between the phase transition and the homogenization process.

The transition layer between the two phases has a thickness of size $\e$, which is the same scale of the micro-structures that form within this layer due to the potential term.
The main challenge of this work will be to handle the situation in which the orientation of the interface is not aligned with the directions of periodicity of the potential $W$.
This misalignment will give rise to the anisotropy in the limiting energy (see Figure \ref{figure:misalignment}).
In particular, the cell problem for the limiting energy density (see Definition \ref{def:sigma}) cannot be reduced to a one dimensional optimal profile problem, as in the case of the energy \eqref{eq:envdw}  (see Figure \ref{fig:layermicrostr}).
This phenomenon is well known in models for solid-solid phase transitions, when higher derivatives are considered in the energy (see, for instance, \cite{ContiFonLeo}).

\begin{figure}
\includegraphics[scale=0.6]{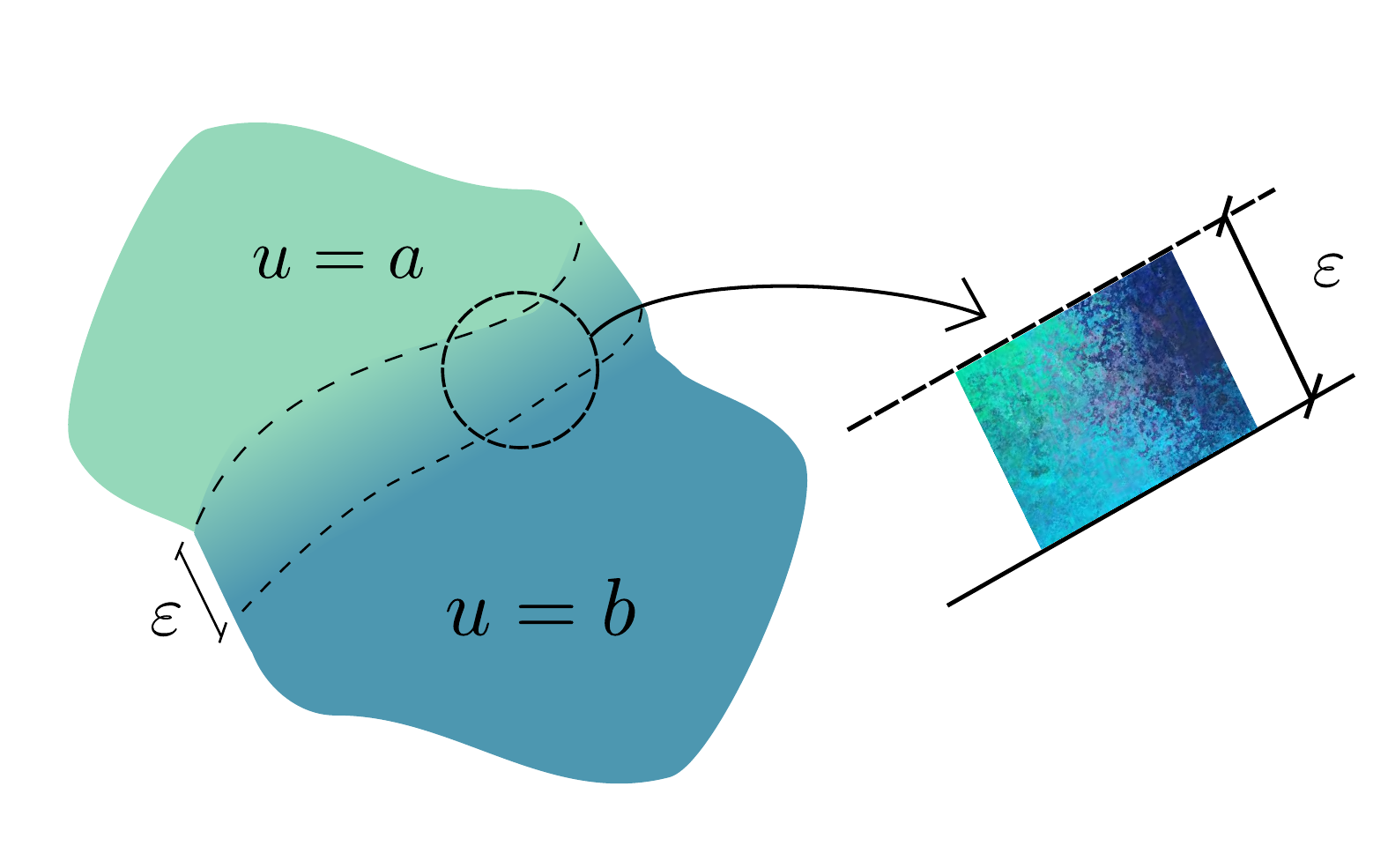}
\caption{In the transition layer of size $\e$ between the region where $u=a$ and $u=b$, microstructures with the same scale $\e$ will develop due to homogenization effects.}
\label{fig:layermicrostr}
\end{figure}

The case where different scalings are present, namely when the small heterogeneities are at a scale $\delta(\e)$ with
$\lim_{\e\to0}\frac{\delta(\e)}{\e}\in\{0,\infty\}$, will be treated in a forthcoming paper.
Moreover, the case in which the wells of the potential $W$ depend on the spatial variable $x$, modeling non-isothermal condition, is currently under investigation.\\

In the literature we can find other problems treating simultaneously phase transitions and homogenization.
In \cite{AnsBraChi2} (see also \cite{AnsBraChi1}) Ansini, Braides and Chiad\`{o} Piat considered the family of functionals
\[
\mathcal{S}_\e(u):=\int_\o \left[\, \frac{1}{\e}W(u(x))+\e f\left( \frac{x}{\delta(\e)}, \nabla u \right)  \,\right] \dx\,,
\]
and identified the $\Gamma$-limit in all three regimes
\begin{equation}\label{eq:regimes}
\lim_{\e\to0}\frac{\e}{\delta(\e)}=0\,,\quad\quad
\lim_{\e\to0}\frac{\e}{\delta(\e)}:=c>0\,,\quad\quad
\lim_{\e\to0}\frac{\e}{\delta(\e)}=+\infty\,,\\
\end{equation}
using abstract $\Gamma$-convergence techniques to prove the general form of the limiting functional, and more explicit arguments to derive the explicit expression in the three regimes (actually, in the first case they need to assume $\e^{3/2}\delta^{-1}(\e)\to0$ as $\e\to0$).

Moreover, we mention the articles \cite{DirMarNov2} and \cite{DirMarNov1} by Dirr, Lucia and Novaga regarding a model for phase transition with an additional bulk term modeling the interaction of the fluid with a periodic mean zero external field.
In \cite{DirMarNov2} they considered, for $\alpha\in(0,1)$, the family of functionals
\[
\mathcal{V}^{(1)}_\e(u):=\int_\o \left[\, \frac{1}{\e}W(u(x)) + \e|\nabla u|^2
    + \frac{1}{\e^\alpha}g\left(\frac{x}{\e^\alpha}\right)u(x)  \,\right] \dx\,,
\]
for some $g\in L^\infty(\o)$, while in \cite{DirMarNov1} they treated the case
\[
\mathcal{V}^{(2)}_\e(u):=\int_\o \left[\, \frac{1}{\e}W(u(x)) + \e|\nabla u|^2
    + \nabla v\left(\frac{x}{\e}\right)\cdot\nabla u(x)  \,\right] \dx\,,
\]
where $v\in W^{1,\infty}(\R^N)$.
Notice that $\mathcal{V}^{(1)}_\e$ is a particular case of $\mathcal{V}^{(2)}_\e$ when $\alpha=1$ and $v\in H^2(\o)$ has vanishing normal derivative on $\partial\o$. An explicit expression of the $\Gamma$-limit is provided in both cases.

The work \cite{BraZep} by Braides and Zeppieri is similar in spirit to the ongoing project of ours where we consider the case of the wells of $W$ depending on the space variable $x$. 
Indeed, in \cite{BraZep} the authors studied the asymptotic behavior of the family of functionals
\[
\mathcal{G}^{(k)}_\e(u):=\int_0^1 \left[ \, W^{(k)}\left(\frac{t}{\delta(\e)},u(x)\right) + \e^2 |u'(t)|^2  \,\right] \dd t\,,
\]
for $\delta(\e)>0$ and the potential $W^{(k)}$ defined, for $k\in[0,1)$, as
\[
W^{(k)}(t,s):=
\left\{
\begin{array}{ll}
W(s-k) & t\in\left(0,\frac{1}{2}\right),\\
W(s+k) & t\in\left(\frac{1}{2},1\right),\\
\end{array}
\right.
\]
with $W(t):=\min\{(t-1)^2, (t+1)^2\}$.
For $k\in(0,1)$ the fact that the zeros of $W^{(k)}$ oscillate at a scale of $\delta(\e)$ leads to the formation of microscopic oscillations, whose effect is studied by identifying the zeroth, the first and the second order $\Gamma$-limit expansions (with the appropriate rescaling) in the three regimes \eqref{eq:regimes}.

In the context of the gradient theory for solid-solid phase transition, we mention the work \cite{FrancMull} by Francfort and M\"{u}ller, where the asymptotic behavior of the energy
\[
\mathcal{L}_\e(u):=\int_\o \left[\, W\left(\frac{x}{\e^\gamma},\nabla u(x)\right)+\e^2 |\triangle u|^2  \,\right] \dx\,.
\]
for $\gamma>0$ is studied under some growth conditions on the potential $W$.

Finally, in \cite{LR} the authors studied the gradient flow of the energy \eqref{eq:energynotrescaled} in the case where the parameter $\e$ in front of the term $|\nabla u|^2$ is kept fixed and only the parameter $\e$ in $W(x/\e,u)$ is allowed to vary.

%%%%%%%%%%%%%%%%%%%%%%%%%%%%%%%%%%%%%%%%%%%%%%%%%%%
%%%%%%%%%%%%%%%%%%%%%%%%%%%%%%%%%%%%%%%%%%%%%%%%%%%

\subsection{Statement of the main result}\label{sec:mainresult}

In the following $Q\subset\R^N$ denotes the unit cube centered at the origin with faces orthogonal to the coordinate axes, $Q:=(-1/2,1/2)^N$.
Consider a double well potential $W:\R^N\times\R^d\to[0,\infty)$ satisfying the following properties:
\begin{itemize}
\item [(H0)] $x\mapsto W(x,p)$ is $Q$-periodic for all $p\in\R^d$,
\item [(H1)] $W$ is a Carath\'{e}odory function, \emph{i.e.},
\begin{itemize}
\item[(i)] for all $p\in\R^d$ the function $x\mapsto W(x,p)$ is measurable, 
\item[(ii)] for a.e. $x\in Q$ the function $p\mapsto W(x,p)$ is continuous,
\end{itemize}
\item [(H2)] there exist $a,b\in\R^d$ such that $W(x,p)=0$ if and only if $p\in\{a,b\}$, for a.e. $x\in Q$,
\item[(H3)] there exists a continuous function $\widetilde{W}:\R^d\to[0,\infty)$ such that $\widetilde{W}(p)\leq W(x,p)$ for a.e. $x\in Q$ and $\widetilde{W}(p)=0$ if and only if $p\in\{a,b\}$.
\item [(H4)] there exist $C>0$ and $q\geq2$ such that $\frac{1}{C}|p|^q-C\leq W(x,p)\leq C(1+|p|^q)$ for a.e. $x\in Q$ and all $p\in\R^d$.
\end{itemize}

\begin{remark}
The choice $q\geq 2$ is connected to the exponent we used in the term $|\nabla u|^2$ of the energy \eqref{eq:energynotrescaled}. If that term is substituted with $|\nabla u|^{\bar{q}}$, in (H4) we would need to take $q\geq\bar{q}$.

Hypotheses (H1), (H2) (H3) and (H4) conform with the prototypical potential
\[
W(x,p):=\sum_{i=1}^k \chi_{E_i}(x)W_i(p)\,,
\]
where $E_i\subset Q$ are measurable pairwise disjoint sets with $Q=\cup_{i=1}^k E_i$, and $W_i:\R^d\to[0,\infty)$ are continuous functions with quadratic growth at infinity and such that $W_i(p)=0$ if and only if $p\in\{a,b\}$, modeling the case of a heterogeneous mixture composed of $k$ different compositions. Here $\widetilde{W}$ in (H3) may be taken as $\widetilde{W}:=\min\{W_1,\dots,W_k\}$.
\end{remark}

Let $\o\subset\R^N$ be an open bounded set with Lipschitz boundary.
For $\e>0$ consider the energy $\f_\e: \sp\to[0,\infty]$ defined as
\begin{equation}\label{eq:energyeps}
\f_\e(u):=\int_\o \left[ \, \frac{1}{\e} W\left(\frac{x}{\e},u(x)\right) + \e |\nabla  u(x)|^2  \,\right] \dx\,,
\end{equation}
where $|\nabla u(x)|$ denotes the Euclidean norm of the $d\times N$ matrix $\nabla u(x)\in\R^{d\times N}$ (matrices with $d$ rows and $N$ columns).

We introduce some definitions.
For $\nu\in\S^{N-1}$, with $\S^{N-1}$ the unit sphere of $\R^N$, we denote by $\mathcal{Q}_\nu$ the family of cubes $Q_\nu$ centered at the origin with two faces orthogonal to $\nu$ and with unit length sides.

\begin{definition}\label{def:mollu0}
Let $\nu\in\S^{N-1}$ and define the function $u_{0,\nu}:\R^N\to\R^d$ as
\begin{equation}\label{eq:u0}
u_{0,\nu}(y):=\left\{
\begin{array}{ll}
a & \text{ if } y\cdot \nu\leq 0\,,\\
b & \text{ if } y\cdot \nu> 0\,.\\
\end{array}
\right.
\end{equation}
Fix a function $\rho\in C^\infty_c(B(0,1))$ with $\int_{\R^N}\rho(x)\dd x=1$, where $B(0,1)$ is the unit ball in $\mathbb{R}^N$.
For $T>0$, set $\rho_T(x):=T^N\rho(Tx)$ and
\begin{equation}\label{eq:u0conv}
\widetilde{u}_{\rho, T, \nu}:=\rho_{T}\ast u_{0,\nu}\,.
\end{equation}
When it is clear from the context, we will abbreviate $\widetilde{u}_{\rho, T,\nu}$ as $\widetilde{u}_{T,\nu}$.
\end{definition}

\begin{definition}\label{def:sigma}
We define the function $\sigma:\S^{N-1}\to[0,\infty)$ as
\[
\sigma(\nu):=\lim_{T\to\infty} g(\nu,T)\,,
\]
where
\[
g(\nu,T):=\frac{1}{T^{N-1}}\inf\Bigl\{\, \int_{T Q_\nu}\left[ W(y,u(y))+|\nabla u|^2  \right]\dd y \,:\,
    Q_\nu\in\mathcal{Q}_\nu,\, u\in \mathcal{C}(\rho, Q_\nu, T)   \,\Bigr\}\,,
\]
and
\[
\mathcal{C}(\rho, Q_\nu, T):=\Bigl\{ u\in H^1(TQ_\nu;\R^d)\,:\, u=\widetilde{u}_{\rho, 1,\nu} \text{ on } \partial\, (TQ_\nu) \,\Bigr\}\,.
\]
Just as before, if there is no possibility of confusion, we will write $\mathcal{C}(\rho, Q_\nu, T)$ as $\mathcal{C}(Q_\nu, T)$.
\end{definition}

\begin{figure}
\includegraphics[scale=0.8]{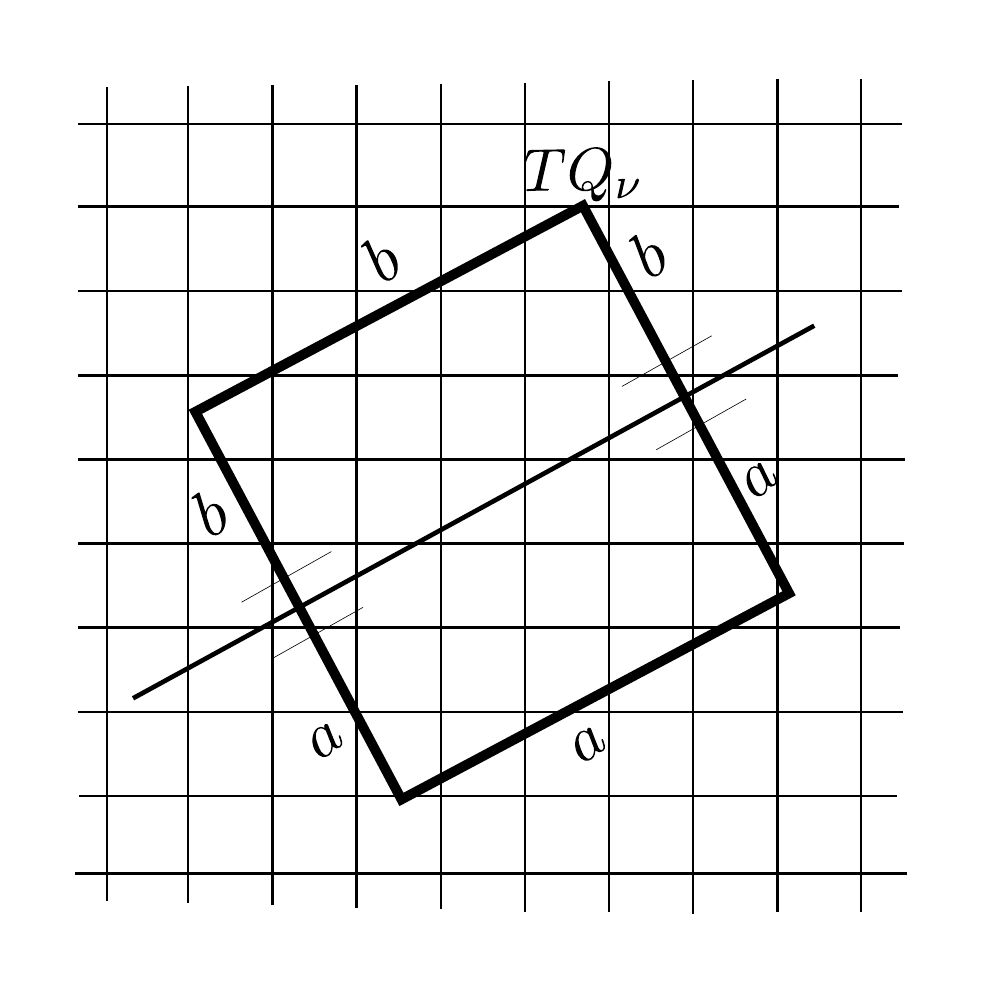}
\caption{The misalignment between a square $Q_\nu$ with two faces orthogonal to $\nu$ and the directions of periodicity of $W$ (the grid in the picture) is the reason for the anisotropy character of the limiting surface energy.}
\label{figure:misalignment}
\end{figure}

\begin{remark}\label{rem:sigma}
For every $\nu\in\S^{N-1}$, $\sigma(\nu)$ is well defined and finite (see Lemma \ref{lem:estimatesigma}) and its definition does not depend on the choice of the mollifier $\rho$ (see Lemma \ref{lem:indepmoll}). Moreover, the function $\nu\mapsto\sigma(\nu)$ is upper semi-continuous on $\S^{N-1}$ (see Proposition \ref{prop:sigma}).

Using \cite{BFM}, it is possible to prove that the infimum in the definition of $g(\nu,T)$ may be taken with respect to one fixed cube $Q_\nu\in\mathcal{Q}_\nu$. Namely, given $\nu\in\S^{N-1}$ and $Q_\nu\in\mathcal{Q}_\nu$ it holds
\[
\sigma(\nu)=\lim_{T\to\infty}\frac{1}{T^{N-1}}
    \inf\Bigl\{\, \int_{T Q_\nu}\left[ W(y,u(y))+|\nabla  u|^2  \right]\dd y \,:\,
    u\in \mathcal{C}(Q_\nu, T)   \,\Bigr\}\,.
\]
\end{remark}

\begin{remark}
In the context of homogenization when dealing with nonconvex potentials $W$ it is natural to consider, in the cell problem for the limiting density function $\sigma$, the infimum over all possible cubes $TQ_\nu$.
For instance, this was observed by M\"{u}ller in \cite{Mul}, where the asymptotic behavior as $\e\to\infty$ of the family of functionals
\[
G_\e(u):=\int_{\o} W\left( \frac{x}{\e},\nabla u \right)\ dx,
\]
defined for $u\in H^1(\o;\R^d)$, is studied. The limiting energy is of the form
\[
\int_\o \overline{W}(\nabla u(x))\dx,
\]
with
\[
\overline{W}(\lambda):=\inf_{k\in\N}\inf_{\psi\in W^{1,p}_0(kQ)}\frac{1}{k^N}\int_{kQ} W(y, \lambda+\nabla\psi(y))\dy.
\]
In the case where $W$ is convex, the infimum over $k\in\N$ is not needed (see \cite{Marc}).
\end{remark}

Consider the functional $\f_0:L^1(\o;\R^d)\to[0,\infty]$ defined by
\begin{equation}\label{eq:limitenergy}
\f_0(u):=
\left\{
\begin{array}{ll}
\displaystyle\int_{\partial^*A} \sigma(\nu_A(x))\,\dd\hno(x) & \text{ if } u\in BV(\o;\{a,b\}),\\
&\\
+\infty & \text{ else},
\end{array}
\right.
\end{equation}
where $A:=\{u=a\}$ and $\nu_A(x)$ denotes the measure theoretic external unit normal to the reduced boundary $\partial^*A$ of $A$ at $x$ (see Definition \ref{def:mun}).\\

We now state the main result of this paper that ensures compactness of energy bounded sequences and identifies the asymptotic behavior of the energies $\f_\e$.

\begin{theorem}\label{thm:main}
Let $\{\e_n\}_{n\in\N}$ be a sequence such that $\e_n\to0$ as $n\to\infty$.
Assume that (H0), (H1), (H2), (H3) and (H4) hold.
\begin{itemize}
\item[(i)] If $\{u_n\}_{n\in\N}\subset \sp$ is such that
\[
\sup_{n\in\N}\f_{\e_n}(u_n)<+\infty
\]
then, up to a subsequence (not relabeled), $u_n\to u$ in $L^1(\o;\R^d)$, where $u\in BV(\o;\{a,b\})$,
\item[(ii)]  As $n\to\infty$, it holds $\f_{\e_n}\stackrel{\Gamma-L^1}{\longrightarrow}\f_0$.
\end{itemize}
Moreover, the function $\sigma:\S^{N-1}\to[0,\infty)$ is continuous.
\end{theorem}

\begin{remark}
The limiting functional $\f_0$ is an anisotropic perimeter functional, whose limiting energy density $\sigma$ is defined via a cell problem describing the intricate interaction between homogenization and phases transition.
It is interesting to notice that in phase transitions models of the form
\[
\int_\o\left[\, \frac{1}{\e}W(u(x)) +\e h(\nabla u(x)) \,\right] \dx
\]
one would expect the limiting model to be isotropic if $h$ is.
Instead, in our case, the anisotropy originates from the mismatch between the square $Q$ related to the periodicity of $W(\cdot,p)$, and a square having two faces orthogonal to the normal $\nu$ to the interface.
\end{remark}

Once Theorem \ref{thm:main} is established, using well known arguments to deal with the mass constraint (see \cite{Modica}) and the result by Kohn and Sternberg (\cite{KohnStern}) for approximating isolated local minimizers, we also obtain the following.

\begin{corollary}\label{corollary}
Let $m\in(0,|\o|)$ and consider, for $\e>0$, the functionals $\mathcal{G}_\e:L^1(\o;\R^d)\to[0,+\infty]$ given by
\[
\mathcal{G}_\e(u):=
\left\{
\begin{array}{ll}
\mathcal{F}_\e(u) & \text{ if } \int_\o u(x)\dx=ma+(|\o|-m)b,\\
&\\
+\infty & \text{ otherwise }.
\end{array}
\right.
\]
Under the assumptions of Theorem \ref{thm:main} it holds that $\mathcal{G}_\e\stackrel{\Gamma-L^1}{\longrightarrow}\mathcal{G}_0$, where $\mathcal{G}_0:L^1(\o;\R^d)\to[0,+\infty]$ is given by
\[
\mathcal{G}_0(u):=
\left\{
\begin{array}{ll}
F_0(u) & \text{ if } \int_\o u(x)\dx=ma+(|\o|-m)b,\\
&\\
+\infty & \text{ otherwise }.
\end{array}
\right.
\]
In particular, every cluster point of a sequence of $\e$-minimizers for $\{\mathcal{G}_\e\}_{\e>0}$ is a minimizer for $\mathcal{G}_0$, and, moreover, every isolated local minimizer $u$ of $\mathcal{G}_0$ can be obtained as the $L^1$ limit of $\{u_\e\}_{\e>0}$, where $u_\e$ is a local minimizer of $\mathcal{G}_\e$.
\end{corollary}

The proof of the Theorem \ref{thm:main} will be divided in several parts.
We would like to briefly comment on the main ideas we will use.

After recalling some preliminary concepts in Section \ref{sec:prel} and establishing auxiliary technical results in Section \ref{sec:techres},
we will prove the compactness result of Theorem \ref{thm:main} (i) (see Proposition \ref{prop:comp}) by reducing our functional to the standard Cahn-Hilliard energy \eqref{eq:envdw}.

In Proposition \ref{prop:liminf} we will obtain the liminf inequality by using the blow-up method introduced by Fonseca and M\"{u}ller in \cite{FonMul0} (see also \cite{FonMul}). Although this strategy can nowadays be considered standard, for clarity and completeness we include the argument.

The limsup inequality is presented in Proposition \ref{prop:limsup} and requires new geometric ideas.
This is due to the fact that the periodicity of $W$ in the first variable is an essential ingredient to build a recovery sequence.
It turns out (see Proposition \ref{prop:periodW}) that there exists a dense set $\Lambda\subset\S^{N-1}$ such that, for every $v_1\in\Lambda$ there exists $T_{v_1}\in\N$ and $v_2,\dots,v_N\in\Lambda$ for which $W(x+T_{v_1} v_i,p)=W(x,p)$ for a.e. $x\in\o$, all $p\in\R^N$ and all $i=1,\dots,N$, and such that $\{v_1,\dots,v_N\}$ is an orthonormal basis of $\R^N$.
Using this fact, in the first step of the proof of Proposition \ref{prop:limsup} we obtain a recovery sequence for the special class of functions $u\in BV(\o;\{a,b\})$ for which the normals to the interface $\partial^* A$, where $A:=\{u=a\}$, belong to $\Lambda$.
We decided to construct a recovery sequence only locally, in order to avoid the technical problem of gluing together optimal profiles for different normal directions to the transition layer. For this reason, we first prove that the localized version of the $\Gamma$-limit is a Radon measure absolutely continuous with respect to $\hno\restr\partial^* A$, and then we show that its density, identified using cubes whose faces are orthogonal to elements of $\Lambda$, is bounded above by $\sigma$.
Finally, in the second step we conclude using a density argument that will invoke Reshetnyak's upper semi-continuity theorem (see Theorem \ref{thm:resh}) and the upper semi-continuity of $\sigma$ (see Proposition \ref{prop:sigma}).

%%%%%%%%%%%%%%%%%%%%%%%%%%%%%%%%%%%%%%%%%%%%%%%%
%%%%%%%%%%%%%%%%%%%%%%%%%%%%%%%%%%%%%%%%%%%%%%%%
%%%%%%%%%%%%%%%%%%%%%%%%%%%%%%%%%%%%%%%%%%%%%%%%
%%%%%%%%%%%%%%%%%%%%%%%%%%%%%%%%%%%%%%%%%%%%%%%%

\section{Preliminaries}\label{sec:prel}

In this section we collect basic notions needed in the paper.

%%%%%%%%%%%%%%%%%%%%%%%%%%%%%%%%%%%%%%%%%%%%%%%%
%%%%%%%%%%%%%%%%%%%%%%%%%%%%%%%%%%%%%%%%%%%%%%%%

\subsection{Finite nonnegative Radon measures}

The family of finite nonnegative Radon measures on a topological space $(X,\tau)$ will be denoted by $\mathcal{M}(X)$.

\begin{definition}
Let $(X,\mathrm{d})$ be a $\sigma$-compact metric space. We say that a sequence $\{\mu_n\}_{n\in\N}\subset\mathcal{M}(X)$ \emph{weakly-${*}$ converges} to a finite nonnegative Radon measure $\mu$ if
\[
\int_X \varphi \,\dd\mu_n\to\int_X \varphi \,\dd\mu
\]
as $n\to\infty$, for all $\varphi\in C_0(X)$, where $C_0(X)$ is the completion in the $L^\infty$ norm of the space of continuous functions with compact support on $X$. In this case we write $\mu_n\stackrel{*}{\rightharpoonup}\mu$.
\end{definition}

The following compactness result for Radon measures is well known (see \cite[Proposition 1.202]{FonLeo}).

\begin{theorem}
Let $(X,\mathrm{d})$ be a $\sigma$-compact metric space and let $\{\mu_n\}_{n\in\N}\subset\mathcal{M}(X)$ be such that
$\sup_{n\in\N}\mu_n(X)<\infty$. Then the exist a subsequence (not relabeled) and $\mu\in\mathcal{M}(X)$ such that $\mu_n\stackrel{*}{\rightharpoonup}\mu$.
\end{theorem}

%%%%%%%%%%%%%%%%%%%%%%%%%%%%%%%%%%%%%%%%%%%%%%%%
%%%%%%%%%%%%%%%%%%%%%%%%%%%%%%%%%%%%%%%%%%%%%%%

\subsection{Sets of finite perimeter}

We recall the definition and some well known facts about sets of finite perimeter (we refer the reader to \cite{AFP} for more details).

\begin{definition}
Let $E\subset\R^N$ with $|E|<\infty$ and let $\o\subset\R^N$ be an open set.
We say that $E$ has \emph{finite perimeter} in $\o$ if
\[
P(E;\o):=\sup\left\{\, \int_E \mathrm{div}\varphi \,\dd x \,:\, \varphi\in C^1_c(\o;\R^N)\,,\, \|\varphi\|_{L^\infty}\leq1  \,\right\}<\infty\,.
\]
\vspace{0\baselineskip}
\end{definition}

\begin{remark}\label{rem:defvar}
$E\subset\R^N$ is a set of finite perimeter in $\o$ if and only if $\chi_E\in BV(\o)$, \emph{i.e.}, the distributional derivative $D\chi_E$ is a finite vector valued Radon measure in $\o$, with
\[
\int_{\R^N} \varphi \,\dd D\chi_E=\int_E \mathrm{div}\varphi \,\dd x
\]
for all $\varphi\in C^1_c(\o;\R^N)$, and $|D\chi_E|(\o)=P(E;\o)$. \vspace{0.8\baselineskip}
\end{remark}

\begin{remark}
Let $\o\subset\R^N$ be an open set, let $a,b\in\R^d$, and let $u\in L^1(\o;\{a,b\})$.
Then $u$ is a function of \emph{bounded variation} in $\o$, and we write $u\in BV(\o;\{a,b\})$, if the set $\{u=a\}:=\{ x\in \o \,:\, u(x)=a\}$ has finite perimeter in $\o$. \vspace{0.5\baselineskip}
\end{remark}

\begin{definition}\label{def:mun}
Let $E\subset\R^N$ be a set of finite perimeter in the open set $\o\subset\R^N$. We define $\partial^* E$, the \emph{reduced boundary} of $E$, as the set of points $x\in\R^N$ for which the limit
\[
\nu_E(x):=-\lim_{r\to0}\frac{D\chi_E(x+rQ)}{|D\chi_E|(x+rQ)}
\]
exists and is such that $|\nu_E(x)|=1$.
The vector $\nu_E(x)$ is called the \emph{measure theoretic exterior normal} to $E$ at $x$. \vspace{0.5\baselineskip}
\end{definition}

We now recall the structure theorem for sets of finite perimeter due to De Giorgi (see~\cite[Theorem 3.59]{AFP} for a proof of the following theorem).

\begin{theorem}\label{thm:DeGiorgi}
Let $E\subset\R^N$ be a set of finite perimeter in the open set $\o\subset\R^N$.
Then
\begin{itemize}
\item[(i)] for all $x\in\partial^* E$ the set $E_r:=\frac{E-x}{r}$ converges locally in $L^1(\R^N)$ as $r\to0$ to the 
    halfspace orthogonal to $\nu_E(x)$ and not containing $\nu_E(x)$,
\item[(ii)] $D\chi_E=-\nu_E\,\hno\restr\partial^* E$,
\item[(iii)] the reduced boundary $\partial^* E$ is $\hno$-rectifiable, \emph{i.e.},
there exist Lipschitz functions $f_i:\R^{N-1}\to\R^N$, $i\in\N$, such that
\[
\partial^* E=\bigcup_{i=1}^\infty f_i(K_i)\,,
\]
where each $K_i\subset\R^{N-1}$ is a compact set.
\end{itemize}
\vspace{0\baselineskip}
\end{theorem}

\begin{remark}\label{rem:newdefnormal}
Using the above result it is possible to prove that (see \cite[Proposition 2.2]{ADM})
\[
\nu_E(x)=-\lim_{r\to0}\frac{D\chi_E(x+rQ)}{r^{N-1}}
\]
for all $x\in\partial^* E$.
\end{remark}

Finally, we state a result due to Reshetnyak in the form we will need in this paper (for a statement and proof of the general case see, for instance, \cite[Theorem 2.38]{AFP}).

\begin{theorem}\label{thm:resh}
Let $\{E_n\}_{n=1}^\infty$ be a sequence of sets of finite perimeter in the open set $\o\subset\R^N$ such that $D\chi_{E_n}\stackrel{*}{\rightharpoonup} D\chi_E$ and $|D\chi_{E_n}|(\o)\to|D\chi_{E}|(\o)$, where $E$ is a set of finite perimeter in $\o$.
Let $f:\mathbb{S}^{N-1}\to[0,\infty)$ be an upper semi-continuous bounded function. Then
\[
\limsup_{n\to\infty}\int_{\partial^* E_n\cap \o} f\left(\nu_{E_n}(x)\right)\,\dd\hno(x)
    \leq \int_{\partial^* E\cap \o} f\left(\nu_E(x)\right)\,\dd\hno(x)\,.
\]
\vspace{0.5\baselineskip}
\end{theorem}

%%%%%%%%%%%%%%%%%%%%%%%%%%%%%%%%%%%%%%%%%%%%%%%%
%%%%%%%%%%%%%%%%%%%%%%%%%%%%%%%%%%%%%%%%%%%%%%%%

\subsection{$\Gamma$-convergence}

We refer to \cite{Braides} and \cite{dalmaso93} for a complete study of $\Gamma$-convergence in metric spaces.

\begin{definition}\label{def:gc}
Let $(X,\mathrm{m})$ be a metric space. We say that $F_n:X\to[-\infty,+\infty]$ $\Gamma$-converges to $F:X\to[-\infty,+\infty]$,
and we write $F_n\stackrel{\Gamma-\mathrm{m}}{\longrightarrow} F$, if the following hold:
\begin{itemize}
\item[(i)] for every $x\in X$ and every $x_n\to x$ we have
\[
F(x)\leq\liminf_{n\to\infty} F_n(x_n)\,,
\]
\item[(ii)] for every $x\in X$ there exists $\{x_n\}_{n=1}^\infty\subset X$ (so called \emph{a recovery sequence}) with $x_n\to x$ such that
\[
\limsup_{n\to\infty} F_n(x_n)\leq F(x)\,.
\]
\end{itemize} \vspace{0\baselineskip}
\end{definition}

In the proof of the limsup inequality we will need to show that a certain set function is actually (the restriction to the family of open sets of) a finite Radon measure. The classical way to prove this is by using the De Giorgi-Letta coincidence criterion (see \cite{DGLetta}), namely to show that the set function is inner regular as well as super and sub additive.
In this paper we will use a simplified coincidence criterion due to Dal Maso, Fonseca and Leoni (see \cite[Corollary 5.2]{DMFonLeo}).
Given $\o\subset\R^N$ an open set, we denote by $\cA(\o)$ the family of all open subsets of $\o$.

\begin{lemma}\label{lem:coincidence}
Let $\lambda:\cA(\o)\to[0,\infty)$ be an increasing set function such that:
\begin{enumerate}
\item[(i)] for all $U,V,W\in\cA(\o)$ with $\overline{U} \subset V \subset W$ it holds
\[
\lambda(W) \leq \lambda(W \setminus \overline{U}) +\lambda(V)\,,
\]
\item[(ii)] $\lambda(U\cap V)=\lambda(U)+\lambda(V)$, for all $U,V\in\cA(\o)$ with $U\cap V=\emptyset$,
\item[(iii)] there exists a measure $\mu:\mathcal{B}(\o)\to[0,\infty)$ such that
\[
\lambda(U)\leq\mu(U)
\]
for all $U\in\cA(\o)$, where $\mathcal{B}(\o)$ denotes the family of Borel sets of $\o$.
\end{enumerate}
Then $\lambda$ is the restriction to $\cA(\o)$ of a measure defined on $\mathcal{B}(\o)$.
\end{lemma}

%%%%%%%%%%%%%%%%%%%%%%%%%%%%%%%%%%%%%%%%%%%%%%%%
%%%%%%%%%%%%%%%%%%%%%%%%%%%%%%%%%%%%%%%%%%%%%%%%
%%%%%%%%%%%%%%%%%%%%%%%%%%%%%%%%%%%%%%%%%%%%%%%%
%%%%%%%%%%%%%%%%%%%%%%%%%%%%%%%%%%%%%%%%%%%%%%%%

\section{Preliminary technical results}\label{sec:techres}

The first result relies on De Giorgi's slicing method (see \cite{DG}), and it allows to adjust the boundary conditions of a given sequence of functions without increasing the energy, by carefully selecting where to make the transition from the given function to one with the right boundary conditions.
Although the argument is nowadays considered to be standard, we include it here for the convenience of the reader.

For $\e>0$, we localize the functional $\f_\e$ by setting
\[
\f_\e(u,A):=\int_A \left[ \, \frac{1}{\e} W\left(\frac{x}{\e},u(x)\right) + \e |\nabla  u(x)|^2  \,\right] \dx\,,
\]
where $A\in\mathcal{A}(\o)$ and $u\in H^1(A;\R^d)$.
Also, for $j\in\N$, we define
\[
A^{(j)}:=\{ x\in A \,:\, \mathrm{d}(x,\partial A)<1/j \}\,.
\]

\begin{lemma}\label{lem:DeGslicing}
Let $D\in\mathcal{A}(\o)$ be a cube with $0\in D$ and let $\nu\in\S^{N-1}$.
Let $\{D_k\}_{k\in\N}\subset\mathcal{A}(\o)$ with $D_k\subset D$ be cubes,
%with $\sup_{n\in\N}\hno(\partial E_k)<\infty$,
let $\{\eta_k\}_{k\in\N}$ with $\eta_k\to0$ as $k\to\infty$, and let
$u_k\in H^1(D_k;\R^d)$, with $k\in\N$, satisfy
\begin{itemize}
\item[(i)] $\chi_{D_k}\to\chi_D$ in $L^1(\R^N)$,
\item[(ii)] $u_k\chi_{D_k}\to u_{0,\nu}$ in $L^1(D;\R^d)$,
\item[(iii)] $\sup_{k\in\N}\f_{\eta_k}(u_k,D_k)<\infty$.
\end{itemize}
Let $\rho\in C^\infty_c(B(0,1))$ with $\int_{\R^N} \rho(x)\dd x=1$.
Then there exists a sequence $\{w_k\}_{k\in\N}\subset H^1(D;\R^d)$, with $w_k=\widetilde{u}_{\rho,1/\eta_k,\nu}$ in $D_k^{(j_k)}$, where $\widetilde{u}_{\rho, 1/\eta_k,\nu}$ is defined as in \eqref{eq:u0conv}, for some $\{j_k\}_{k\in\N}$ with $j_k\to\infty$ as $k\to\infty$, such that
\[
\liminf_{k\to\infty}\f_{\eta_k}(u_k,D_k)\geq \limsup_{k\to\infty}\f_{\eta_k}(w_k,D)\,.
\]
Moreover, $w_k\to u_{0,\nu}$ in $L^q(D;\R^d)$ as $k\to\infty$, where $q\geq2$ is as in (H4).
\end{lemma}

\begin{proof}
Assume, without loss of generality, that
\begin{equation}\label{eq:limliminf}
\liminf_{k\to\infty}\f_{\eta_k}(u_k,D_k)=\lim_{k\to\infty}\f_{\eta_k}(u_k,D_k)<+\infty
\end{equation}
and that, as $n\to\infty$, $u_n(x)\chi_{D_k}(x)\to u_{0,\nu}(x)$ for a.e. $x\in D$.\\

\emph{Step 1.} We claim that
\begin{equation}\label{eq:claimslicing}
\lim_{k\to\infty}\|u_k-u_{0,\nu}\|_{L^q(D_k;\R^d)}=0.
\end{equation}
Indeed, using (H4), we get
\begin{equation}\label{eq:ptwiseest}
|u_k(x)-u_{0,\nu}(x)|^q\leq C\left( W\left( \frac{x}{\eta_k}, u_k(x) \right) +1 \right),
\end{equation}
for $x\in D_k$. From \eqref{eq:limliminf} we have $\chi_{D_k}(x)W(\frac{x}{\eta_k},u_k(x))\to0$ as $k\to\infty$
for a.e. $x\in D$, and thus
\begin{align*}
C|D|&-\limsup_{k\to\infty}\|u_k-u_{0,\nu}\|^q_{L^q(D_k;\R^d)}\\
    &=\liminf_{k\to\infty}\int_{D_k} \left[\, CW\left( \frac{x}{\eta_k}, u_k(x) \right) + C - |u_k(x)-u_{0,\nu}(x)|^q \,\right] \dx \\
&\geq \int_{D} \liminf_{k\to\infty} \chi_{D_k}(x)\left[\, CW\left( \frac{x}{\eta_k}, u_k(x) \right) + C - |u_k(x)-u_{0,\nu}(x)|^q \,\right] \dx \\
&\geq C|D|,
\end{align*}
where we used Fatou's lemma and \eqref{eq:ptwiseest}.\\

\emph{Step 2.} Here we abbreviate $\widetilde{u}_{\rho,1/k,\nu}$ as $\widetilde{u}_{1/k,\nu}$.
Set $v_k:=\widetilde{u}_{1/\eta_k,\nu}$ and $\lambda_k:=\|u_k\chi_{D_k}-v_k\|_{L^2(D;\R^d)}$.
Using Step 1, since $q\geq2$ we get $\lim_{k\to\infty}\lambda_k=0$. 
For every $k,j\in\N$ divide $D_k^{(j)}$ into $M_{k,j}$ equidistant layers $L_{k,j}^i$ of width $\eta_k \lambda_k$, for $i=1,\dots,M_{k,j}$.
It holds
\begin{equation}\label{eq:rate}
M_{k,j}\eta_k\lambda_k= \frac{1}{j}\,.
\end{equation}
For every $k,j\in\N$ let $L^{i_0}_{k,j}$, with $i_0\in\{1,\dots,M_{k,j}\}$, be such that
\begin{equation}\label{eq:bestlayer}
\int_{L^{i_0}_{k,j}}  a_{k}(x)\dx \leq\frac{1}{M_{k,j}} \int_{D_k^{(j)}} a_{k}(x) \dx\,,
\end{equation}
where
\[
a_{k}(x):=\frac{1}{\eta_k}(1+|u_k-v_k|^q+|v_k|^q)+\frac{1}{\lambda^2_k\eta_k}|u_k(x)-v_k(x)|^2
    +\eta_k\left(|\nabla u_k(x)|^2 +|\nabla v_k(x)|^2\right) \,.
\]
Further, consider cut-off functions $\varphi_{k,j}\in C^\infty_c(D)$ with
\begin{equation}\label{eq:propphi1}
0\leq \varphi_{k,j}\leq 1\,,\quad\quad\quad \|\nabla \varphi_{k,j}\|\leq \frac{C}{\eta_k\lambda_k}\,,
\end{equation}
such that
\begin{equation}\label{eq:propphi2}
\varphi_{k,j}(x)=1\,,\quad\quad \text{ for } x\in \left(\, \bigcup_{i=1}^{i_0-1} L^i_{k,j} \,\right)
    \cup (D_k\setminus D_k^{(j)})\,,
\end{equation}
\begin{equation}\label{eq:propphi3}
\varphi_{k,j}(x)=0\,,\quad\quad \text{ for } x\in \left(\, \bigcup_{i=i_0+1}^{M_{k,j}} L^i_{k,j} \,\right)
    \cup (D\setminus D_k)\,.
\end{equation}
Set
\[
\widetilde{w}_{k,j}:= \varphi_{k,j} u_k + (1-\varphi_{k,j})v_k\,.
\]
It holds that $\lim_{j\to\infty}\lim_{k\to\infty}\|\widetilde{w}_{k,j}-u_{0,\nu}\|_{L^q(D;\R^d)}=0$.
Let $j_k\in\N$ be such that $ D_k^{(j_k)}\subset \bigcup_{i=i_0+1}^{M_{k,j}} L^i_{k_j}$.
Then $\widetilde{w}_{k,j}=v_k$ in $D_k^{(j_k)}$.
We claim that
\begin{align}\label{eq:limsuppart1}
\liminf_{k\to\infty} \f_{\eta_k}(u_k, D_k)\geq \limsup_{j\to\infty}\limsup_{k\to\infty}\f_{\eta_k}(\widetilde{w}_{k,j}, D)\,.
\end{align}
Indeed
\begin{align}\label{eq:ineqabc}
\f_{\eta_k}(\widetilde{w}_{k,j}, D_k) &=
\f_{\eta_k}\left(\, u_k, \left(\, \bigcup_{i=1}^{i_0-1} L^i_{k,j} \,\right)
    \cup (D_k\setminus D_k^{(j)}) \,\right)+\f_{\eta_k}\left(\, \widetilde{w}_{k,j}, L^{i_0}_{k,j} \,\right) \nonumber\\
&\hspace{1cm} +\f_{\eta_k}\left(\, v_k, \bigcup_{i=i_0+1}^{M_{k,j}} L^i_{k,j} \,\right) \nonumber\\
&=: A_{k,j}+B_{k,j}+C_{k,j}\,.
\end{align}
To estimate the first term in \eqref{eq:ineqabc} we notice that
\begin{align}\label{eq:akj}
&\liminf_{k\to\infty} \f_{\eta_k}(u_k, D_k)
\geq \limsup_{j\to\infty}\limsup_{k\to\infty} A_{k,j}\,.
\end{align}
Consider the term $B_{k,j}$. Using (H4) together with \eqref{eq:propphi1} we have that
\begin{align}\label{eq:bkjpart1}
B_{k,j} &\leq C\int_{L^{i_0}_{k,j}}  \left[\, \frac{1}{\eta_k}(1+|\widetilde{w}_{k,j}|^q) +  \eta_k\left( |\nabla\varphi_{k,j}|^2|u_k-v_k|^2+|\nabla u_k|^2 + |\nabla v_k|^2 \right)  \,\right] \dd x  \nonumber\\
&\leq C\int_{L^{i_0}_{k,j}} \left[\, \frac{1}{\eta_k}(1+|u_k-v_k|^q+|v_k|^q)
    + \frac{1}{\eta_k\lambda_k^2}|u_k-v_k|^2
    +\eta_k\left(|\nabla u_k|^2 + |\nabla v_k|^2 \right)   \,\right] \dd x \nonumber \\
&\leq \frac{C}{M_{k,j}}\int_{D^{(j)}_k} \Biggl[\, \frac{1+|u_k-v_k|^q}{\eta_k} + \frac{|u_k-v_k|^2}{\eta_k\lambda_k^2} +\eta_k\left(|\nabla u_k|^2 + |\nabla v_k|^2 \right)   \,\Biggr] \dd x\,,
\end{align}
where in the last step we used \eqref{eq:bestlayer} and the fact that $\sup_{k\in\N}\|v_k\|_{L^\infty(D;\R^d)}<\infty$.
Since for a cube $rQ$ with side length $r$ we have
\[
|(rQ)^{(j)}| \leq \frac{2N r^{N-1}}{j},
\]
and the cubes $D_k$ are all contained in the bounded cube $D$, we can find $\bar{j}\in\N$ such that for all $j\geq\bar{j}$ and $k\in\N$ we get
\begin{align}\label{eq:bkjpart1a}
\frac{|D^{(j)}_k|}{M_{k,j}\eta_k} \leq \frac{C}{jM_{k,j}\eta_k} = C\lambda_k\,.
\end{align}
Step 1 (see \eqref{eq:claimslicing}) yields
\begin{align}\label{eq:bkjpart1b}
\frac{C}{M_{k,j}\eta_k}\int_{D^{(j)}_k} \left[\, 1+|u_k-v_k|^q \,\right] \dx&
\leq  Cj\lambda_k \left[\, \|u_k-v_k\|^q_{L^q(D_k;\R^d)} + 1 \,\right] \leq  Cj\lambda_k\,.
\end{align}
Moreover, by \eqref{eq:rate} we obtain
\begin{equation}\label{eq:bkjpart2}
\frac{1}{M_{k,j}\eta_k\lambda_k^2}\int_{D^{(j)}_k} |u_k-v_k|^2 \dx \leq C j\lambda_k\,,
\end{equation}
\begin{equation}\label{eq:bkjpart3}
\eta_k \int_{D_k}  |\nabla u_k|^2 \dy\leq\limsup_{k\to\infty} \f_{\eta_k}(u_k,D_k)<\infty\,,
\end{equation}
and, since
\begin{equation}\label{eq:estgradvk}
\|\nabla v_k\|_{L^\infty}\leq \frac{C}{\eta_k},
\end{equation}
\begin{equation}\label{eq:bkjpart4}
\frac{\eta_k}{M_{k,j}}\int_{D_k^{(j)}} |\nabla v_k|^2 \dy \leq \frac{C}{M_{k,j}\eta_k}= Cj\lambda_k\,.
\end{equation}
From \eqref{eq:bkjpart1}, \eqref{eq:bkjpart1a}, \eqref{eq:bkjpart1b}, \eqref{eq:bkjpart2}, \eqref{eq:bkjpart3} and \eqref{eq:bkjpart4} we get
\begin{equation}\label{eq:bkj}
\lim_{j\to\infty}\lim_{k\to\infty} B_{k,j}=0\,.
\end{equation}
We now estimate the term $C_{k,j}$. Using \eqref{eq:estgradvk}, we obtain
\begin{align*}
C_{k,j}&\leq \frac{1}{\eta_k}\int_{\bigcup_{i=i_0+1}^{M_{k,j}} L^i_{k,j}} \left[ \, W(v_k(y))
    + \eta_k^2 |\nabla v_k(y)|^2 \,\right] \dd y \\
&\leq \frac{C}{\eta_k}\left|\, D_k^{(j)} \cap \{ x\in D \,:\, |x\cdot\nu|<\eta_k \}  \,\right|\,,
\end{align*}
and so
\begin{align}\label{eq:ckj}
\lim_{j\to\infty}\lim_{k\to\infty} C_{k,j}=0\,.
\end{align}
Similarly, it holds that
\begin{equation}\label{eq:qnuqk}
\lim_{k\to\infty}\f_{\eta_k}(\widetilde{w}_{k,j}, D\setminus D_k)\leq
    \lim_{k\to\infty}\frac{C}{\eta_k}|(D\setminus D_k)\cap\{ x\in D \,:\, |x\cdot\nu|<\eta_k \}|=0\,.
\end{equation}
Using \eqref{eq:ineqabc}, \eqref{eq:akj}, \eqref{eq:bkj}, \eqref{eq:ckj} and \eqref{eq:qnuqk} we obtain \eqref{eq:limsuppart1}.\\

Applying a diagonalizing argument, it is possible to find an increasing sequence $\{j(k)\}_{k\in\N}$ such that
\[
\lim_{k\to\infty}[B_{k,j(k)}+C_{k,j(k)}+\f_{\eta_k}(\widetilde{w}_{k,j(k)}, D\setminus D_k)]=0\,,
\]
and $\lim_{k\to\infty}\|\widetilde{w}_{k,j(k)}-u_{0,\nu}\|_{L^1(D;\R^d)}=0$.
Thus, the sequence $\{w_k\}_{k\in\N}$, with $w_k:=\widetilde{w}_{k,j(k)}$ satisfies the claim of the lemma.
\end{proof}

\begin{remark}\label{rem:DG}
In the paper we will make use of the basic idea behind the proof of Lemma \ref{lem:DeGslicing} in several occasions.
In particular, it is possible to see that the result of Lemma \ref{lem:DeGslicing} still holds true if
the set $D\subset\R^N$ is a finite union of cubes, and $D_k=D$ for all $k\in\N$.
\end{remark}

The proof of the limsup inequality, Proposition \ref{prop:limsup}, uses periodicity properties of the potential energy $W$.
In particular, we will show that $W$ is periodic in the first variable not only with respect to the canonical set of orthogonal direction, but also with respect to a dense set of orthogonal directions.
In the sequel we will use the notation $\Lambda := \mathbb{Q}^N \cap \S^{N-1}$ and $\{e_1, \dots, e_N\}$ will denote the standard orthonormal basis for $\mathbb{R}^N$.
We first recall the following extension theorem for isometries (for a proof see, for instance, \cite[Theorem 10.2]{Lang}).

\begin{theorem}\label{thm:witt}\textnormal{(Witt's Extension Theorem)}
Let $V$ be a finite dimensional vector space over a field $\mathbb{K}$ with characteristic different from $2$, and let $B$ be a symmetric bilinear form on $V$ with $B(u,u)>0$ for all $u \not = 0$. Let $U, W$ be subspaces of $V$ and let $T: U \rga W$ be an isometry, that is, $B(u,v) = B(Tu, Tv)$ for all $u,v \in U$. Then $T$ can be extended to an isometry from $V$ to $V$.
\end{theorem}

\begin{lemma}\label{lem:orthbasis}
Let $\nu \in \Lambda$. Then there exist a rotation $R_{\nu} : \mathbb{R}^N \rga \mathbb{R}^N$ and $\lambda_{\nu}\in\N$ such that $R_{\nu} e_N = \nu$ and $\lambda_{\nu} R_{\nu} e_i \in \mathbb{Z}^N$ for all $i = 1, \dots, N$. 
\end{lemma}

\begin{proof}
Let $\nu \in \Lambda$ be fixed. Consider the spaces
\[
U:= \textrm{Span}(e_N)\,,\quad\quad W := \textrm{Span}(\nu)
\]
as subspaces of $V := \mathbb{Q}^N$ over the field $\mathbb{K} := \mathbb{Q}$, with $B$ being the standard Euclidean inner product. Then, the linear map $T : U \rga W$ defined by $T(e_N):=\nu$ is an isometry.
Apply Theorem \ref{thm:witt} to extend $T$ as a linear isometry $T : \mathbb{Q}^N \rga \mathbb{Q}^N$.
In particular, $T(e_i)\cdot T(e_j)=\delta_{ij}$. Up to redefining the sign of $T(e_1)$ so that $\det T > 0$, we can assume $T$ to be a rotation.
Let $\lambda_{\nu} \in \mathbb{N}$ be such that $\lambda_{\nu} T(e_i) \in \mathbb{Z}^N$ for all $i = 1, \dots, N$.
Finally, define $R_{\nu} :\mathbb{R}^N \rga \mathbb{R}^N$ to be the unique continuous extension of $T$ to all of $\mathbb{R}^N$, which is well defined as isometries are uniformly continuous.
\end{proof}

\begin{proposition}\label{prop:periodW}
Let $\nu_N\in\Lambda$. Then there exist $\nu_1,\dots,\nu_{N-1}\in\Lambda$ and $T\in\N$ such that
$\nu_1,\dots,\nu_{N-1},\nu_N$ is an orthonormal basis of $\R^N$, and for a.e. $x\in Q$ it holds
$W(x+T\nu_i,p)=W(x,p)$ for all $p\in\R^d$ and $i = 1, \dots, N$.
\end{proposition}

\begin{proof}
Let $R:\R^N\to\R^N$ be a rotation and let $T := \lambda_{\nu_N} \in\N$ be given by Lemma \ref{lem:orthbasis} relative to $\nu_N$. Set $\nu_i:=Re_i$ for $i=1,\dots,N-1$.
We have that $T\nu_i\in\Z^N$ for all $i=1,\dots,N$.
Fix $i\in\{1,\dots,N\}$ and write $T\nu_i=\sum_{j=1}^N \lambda_j e_j$, for some $\lambda_j\in\Z$.
For $p \in \mathbb{R}^d$, using the periodicity of $W(\cdot, p)$ with respect to the canonical directions, for a.e. $x\in Q$ we have that
\[
W(x+T\nu_i,p)=W\left(x+\sum_{j=1}^N \lambda_j e_j,p\right)=W(x,p).
\]
\end{proof}

In the following, given a linear map $L:\R^N\to\R^N$, we will denote by $\|L\|$ the Euclidean norm of $L$, \emph{i.e.}, $\|L\|^2:=\sum_{i,j=1}^N [L(e_i)\cdot e_j]^2$.
For the sake of notation, we will also define the set of rational rotations $SO(N; \mathbb{Q}) \subset SO(N)$ as the rotations $R \in SO(N)$ such that $Re_i \in \mathbb{Q}^N$ for $i \in \{1, \dots, N\}$.

\begin{lemma}\label{lem:denserotations}
Let $\varepsilon>0$, $\nu\in\Lambda$, and let $S:\R^N\to\R^N$ be a rotation with $S(e_N)=\nu$.
Then there exists a rotation $R\in SO(N; \mathbb{Q})$ such that
$R(e_N)=\nu$ and $\|R-S\|<\e$.
\end{lemma}

\begin{proof}

\emph{Step 1} We claim that $SO(N; \mathbb{Q})$ is dense in $SO(N)$ for every $N \geq 1$.

We proceed by induction on $N$. When $N = 1$, $SO(N)$ consists of the identity, so the claim is trivial. Let $N>1$ be fixed and let $\e > 0$ and $S \in SO(N)$ be arbitrary. By density of $\mathbb{Q}^N \cap \S^{N-1}$, we can find a sequence $\{q_n\}_{n\in\N}\in\Lambda$ with $|q_n| = 1$ such that $q_n \to S( e_N)$ as $n \to \infty$. By Lemma \ref{lem:orthbasis} we can find $R_n \in SO(N; \mathbb{Q})$ such that $R_n(e_N) = q_n$. Since $SO(N)$ is a compact set, we can extract a convergent subsequence (not relabeled) of $\{R_n\}$ such that $R_n \to R \in SO(N)$, with $R(e_N) = \lim_{n \to \infty} R_n (e_N) = S( e_N)$. 

Thus, the rotation $R^{-1} \circ S$ fixes $e_N$ and may be identified with a rotation $T \in SO(N-1)$, i.e., writing $e_i =: (e_i', 0), i = 1, \dots, N-1$, it follows that $R e_i = (T e_i',0), i = 1, \dots, N-1$. By the induction hypotheses, we can find $T' \in SO(N-1 ; \mathbb{Q})$ such that 
\[ 
\|T- T'\| < \frac{\e}{2}.
\]
Define $R' \in SO(N; \Q)$ by
\[
R' e_i :=
\begin{cases}
(T' e_i', 0) & i = 1, \dots, N-1, \\
e_N & i = N.
\end{cases}
\]
Let $n_0$ be so large that 
\[
\|R - R_{n_0} \| < \frac{\e}{2}.
\]
We claim that our desired rotation is $ R_{n_0} \circ R' \in SO(N; \mathbb{Q})$. Indeed,
\begin{align*}
\| R_{n_0} \circ R' - S \| &\leq \| R_{n_0} \circ R' -   R_{n_0} \circ R^{-1} \circ S \| + \|   R_{n_0} \circ R^{-1} \circ S  - S \| \\
& = \|  R' -  R^{-1} \circ S \| + \|R_{n_0}  -  R \|\\
& = \|  T' -  T \| + \|R_{n_0}  -  R \| < \e.
\end{align*}

\emph{Step 2} Let $S \in SO(N)$ with $S(e_N) = \nu$ be given. If $N = 1$, there is nothing else to prove, so we proceed with $N>1$. 

By Lemma \ref{lem:orthbasis} we can find a rotation $R_1 \in SO(N; \mathbb{Q})$ such that $R_1( e_N) = \nu$. Since $R_1^{-1} \circ S$ is a rotation with $(R_1^{-1} \circ S) (e_N) = e_N$, as in Step 1 we can identify $R^{-1} \circ S$ with a rotation $T_1 \in SO(N-1)$. Also by Step 1, $SO(N-1; \mathbb{Q})$ is dense in $SO(N-1)$, so we can find $T_2 \in SO(N-1; \mathbb{Q})$ such that $\|T_2 - T_1\| < \e$. As before, identifying $T_2$ with a rotation $ R_2 \in SO(N; \mathbb{Q})$ that fixes $e_N$, we set $R := R_1 \circ R_2 \in SO(N; \mathbb{Q})$. We have that $(R_1 \circ R_2) (e_N) = R_1 (e_N) = \nu$ and
\[
\|R_1 \circ R_2 - S\| = \| R_2 - R_1^{-1} \circ S  \| =  \|T_2 - T_1 \| < \e.
\]
\end{proof}

\begin{definition}\label{def:poly}
Let $V\subset\S^{N-1}$. We say that a set $E\subset\R^N$ is a \emph{$V$-polyhedral set} if $\partial E$ is a Lipschitz manifold contained in the union of finitely many affine hyperplanes each of which is orthogonal to an element of $V$.
\end{definition}

A variant of well known approximation results of sets of finite perimeter by polyhedral sets yields the following (see \cite[Theorem 3.42]{AFP}).

\begin{lemma}\label{lem:densitysets}
Let $V\subset\S^{N-1}$ be a dense set. If $E$ is a set with finite perimeter in $\o$, then there exists a sequence $\{E_n\}_{n\in\N}$ of $V$-polyhedral sets such that
\[
\lim_{n\to\infty}\|\chi_{E_n}-\chi_E\|_{L^1(\o)}=0\,,\quad\quad\quad
\lim_{n\to\infty}|P(E_n;\o)-P(E;\o)|=0\,.
\]
\end{lemma}

\begin{proof}
Using \cite[Theorem 3.42]{AFP} it is possible to find a family $\{F_n\}_{n\in\N}\subset\R^N$ of polyhedral sets such that 
\[
\|\chi_{F_n}-\chi_E\|_{L^1(\o)}\leq\frac{1}{n}\,,\quad\quad\quad
|P(F_n;\o)-P(E;\o)|\leq\frac{1}{n}\,.
\]
For every $n\in\N$, let $\Gamma_{1}^{(n)},\dots,\Gamma_{s_n}^{(n)}$ be the hyperplanes whose union contains the boundary of $F_n$.
Let $\nu^{(n)}_1,\dots,\nu^{(n)}_{s_n}\in\S^{N-1}$ be such that $\Gamma_i=(\nu_i^{(n)})^\perp$.
Then it is possible to find rotations $R^{(n)}_i:\R^N\to\R^N$ such that $R^{(n)}_i \nu^{(n)}_i\in V$ and, denoting by $E_n\subset\R^N$ the set \emph{enclosed} by the hyperplanes $(R^{(n)}_i \nu^{(n)}_i)^\perp$, we get
\[
\|\chi_{E_n}-\chi_E\|_{L^1(\o)}\leq\frac{2}{n}\,,\quad\quad\quad
|P(E_n;\o)-P(E;\o)|\leq\frac{2}{n}\,.
\]
\end{proof}

%%%%%%%%%%%%%%%%%%%%%%%%%%%%%%%%%%%%%%%%%%%%%%%%
%%%%%%%%%%%%%%%%%%%%%%%%%%%%%%%%%%%%%%%%%%%%%%%%
%%%%%%%%%%%%%%%%%%%%%%%%%%%%%%%%%%%%%%%%%%%%%%%%
%%%%%%%%%%%%%%%%%%%%%%%%%%%%%%%%%%%%%%%%%%%%%%%%

\section{Properties of the function $\sigma$}

The aim of this section is to study properties of the function $\sigma$ introduced in Definition \ref{def:sigma} that we will need in the proof of Proposition \ref{prop:limsup} in order to prove the limsup inequality.

\begin{lemma}\label{lem:estimatesigma}
Let $\nu\in\S^{N-1}$. Then $\sigma(\nu)$ is well defined and is finite.
\end{lemma}

\begin{proof}
Let $\nu \in \S^{N-1}$.
For $T>\sqrt{N}$ let $Q_T \in\mathcal{Q}_{\nu}$ and $u_T\in\mathcal{C}(Q_T,T)$ be such that
\begin{equation}\label{eq:ut}
\frac{1}{T^{N-1}} \int_{T Q_T} W(y, u_T(y)) + |\nabla u_T(y)|^2 dy \leq
    g(T) + \frac{1}{T},
\end{equation}
where, for simplicity of notation, we write $g(T)$ for $g(\nu, T)$.
Let $\{ \nu^{(1)}_T, \dots, \nu^{(N)}_T \}$ be an orthonormal basis of $\R^N$ normal to the faces of $Q_T$ such that $\nu = \nu^{(N)}_T$.
We define an oriented rectangular prism centered at 0 via
\[
P(\alpha, \beta) := \{x \in \R^N : |x \cdot \nu| \leq \beta \textrm{ and } |x \cdot \nu_T^{(i)}| \leq \alpha \textrm{ for } 1 \leq i \leq N-1 \}.
\]

Let $S > T+3+\sqrt{N}$. We claim that for all $m \in \mathbb{N}$ with $2 \leq m < T$, we have 
\begin{equation}\label{eq:estSTsigma}
g(S)\leq g(T)+R(m,S,T)\,,
\end{equation}
where the quantity $R(m,S,T)$ does not depend on $\nu$ and is such that
\[
\lim_{m \to \infty} \lim_{T\to\infty}\lim_{S\to\infty} R(m,S,T)=0.
\]

Note that if this holds then
\[
\limsup_{S \to \infty} g(S) \leq \liminf_{T \to \infty} g(T),
\]
and this ensures the existence of the limit in the definition of $\sigma$. Therefore, the remainder of Step 1 is dedicated to proving \eqref{eq:estSTsigma}.

The idea is to construct a competitor $u_S$ for the infimum problem defining $g(S)$ by taking $\lfloor \frac{S}{T}\rfloor^{N-1}$ copies of $TQ_\nu\cap\nu^\perp$ centered on $\nu^\perp\cap SQ_\nu$ in each of which we define $u_S$ to be (a translation of) $u_T$.
In order to compare the energy of $u_S$ to the energy of $u_T$, we need the copies of the cube $TQ_\nu$ to be integer translations of the original. Moreover, we also have to ensure that the boundary conditions render $u_S$ admissible for the infimum problem defining $g(S)$. For this reason, we need the centers of the translated copies of $TQ_\nu\cap\nu^\perp$ to be close to $\nu^\perp\cap SQ_\nu$ (recall that the mollifiers $\rho_{T,\nu}$ and $\rho_{S,\nu}$ only depend on the direction $\nu$).\\

Set
\[
M_{T,S} := \bigg \lfloor \frac{S-\frac{1}{T}}{T+\sqrt{N}+2} \bigg \rfloor^{N-1}\,,
\]
and notice that
\begin{equation}\label{eq:mst}
\lim_{T\to\infty}\lim_{S\to\infty} \frac{T^{N-1}}{S^{N-1}}M_{T,S}=1\,.
\end{equation}
We can tile $\left( S - \frac{1}{T} \right) Q_T$ with disjoint prisms $\left\{p_i + P\left( T + \sqrt{N} + 2, S - \frac{1}{T}\right) \right\}_{i=1}^{M_{T,S}}$ so that
\[
p_i + P\left( T + \sqrt{N} + 2, S - \frac{1}{T} \right) \subset \left(S-\frac{1}{T} \right) Q_T,
\quad\quad p_i \in \nu^\perp,
\]
for each $i \in \{1, \dots, M_{T,S} \}$. In each cube $p_i + \sqrt{N} Q_T$ we can find $x_i \in \Z^N$ since dist$( \cdot, \Z^N) \leq \sqrt{N}$ in $\R^N$, and we have
\[
x_i + (T+2)Q_T \subset p_i + (T+\sqrt{N}+2)Q_T.
\]
Consider, for $m\in\N$ and $i \in \{1, \dots, M_{T,S} \}$ cut-off functions $\varphi_{m, i} \in C_c(x_i + (T + \frac{1}{m})Q_T; [0,1])$ be such that
\begin{equation}\label{eq:estphimi}
\varphi_{m, i}(x) =
\left\{
\begin{array}{ll}
0 & \textrm{ if } x \in \partial \bigg( x_i + \bigg(T+ \frac{1}{m} \bigg) Q_T  \bigg),\\
&\\
1 & \textrm{ if } x \in x_i + T Q_T,\\
\end{array}
\right.
\quad\quad
\|\nabla \varphi_{m, i}\|_{L^\infty} \leq Cm,
\end{equation}
for some $C>0$.
Define $u_S: SQ_T\to\R^d$ by
\[
u_S(x) :=
\left\{
\begin{array}{l}
u_T(x-x_i)   \hspace{4cm}\textrm{ if } x \in x_i + T Q_T, \\
\\
\varphi_{m, i}(x)(\rho * u_{0,\nu})(x+p_i- x_i)+(1-\varphi_{m, i}(x))(\rho * u_{0,\nu})(x) \\
\hspace{6cm}\textrm{if } x \in (x_i +(T + \frac{1}{m})Q_T) \setminus (x_i + TQ_T),\\
\\
(\rho* u_{0,\nu})(x) \hspace{3.6cm}\textrm{ otherwise.}
\end{array}
\right.
\]

\begin{figure}\label{fig:competitor}
\includegraphics[scale=0.8]{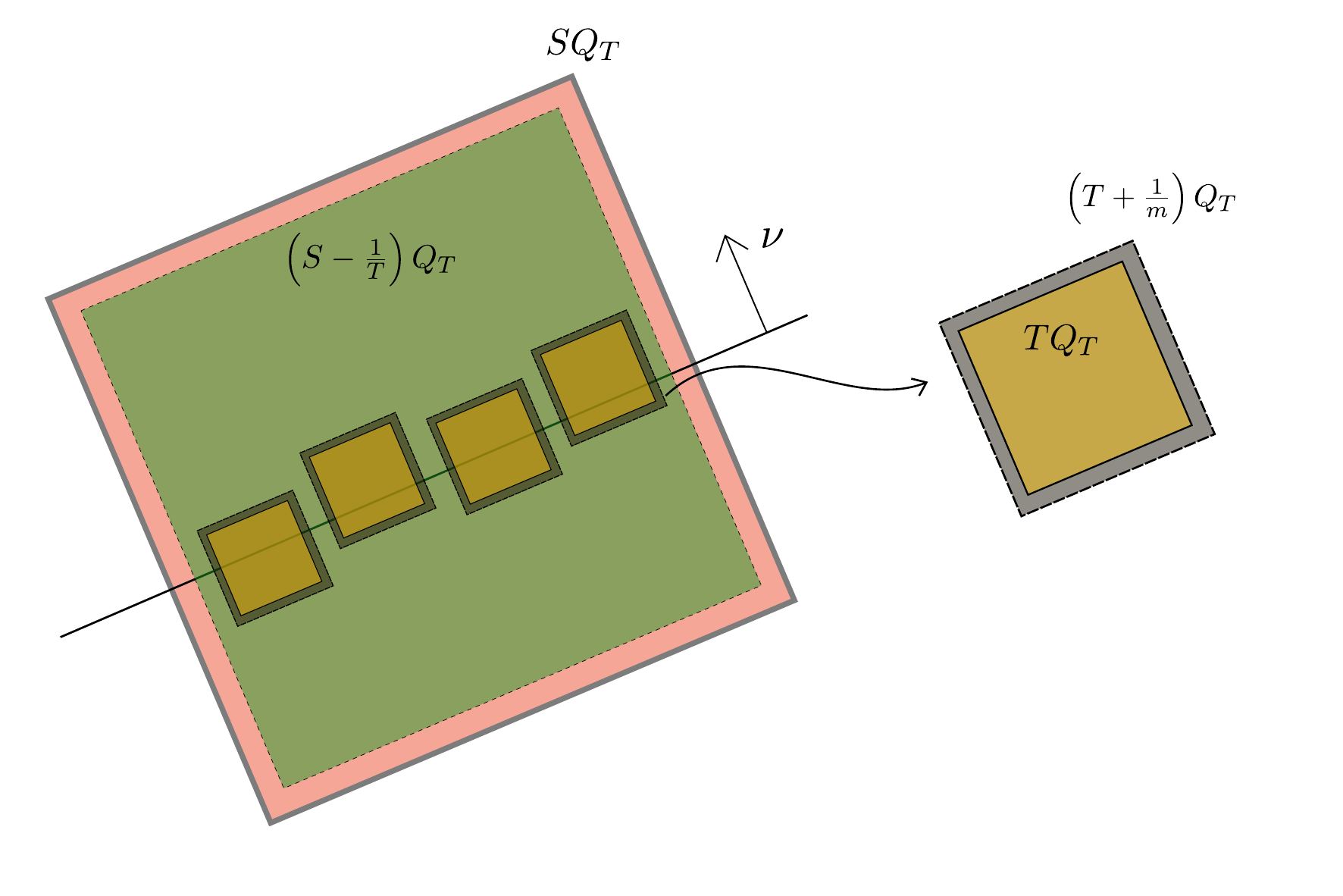}
\caption{Construction of the function $u_S$: in each yellow cube $x_i + TQ_T$ we defined it as a copy of $u_T$ and we use the grey region $(x_i +(T + \frac{1}{m})Q_T) \setminus (x_i + TQ_T)$ around it to adjust the boundary conditions and make them match the value of $u_S$ in the green region.
Finally, in the pink region $SQ_T \setminus (S - \frac{1}{T})Q_T$ we make the transition in order for $u_S$ to be an admissible competitor for the infimum problem defining $g(S)$.}
\end{figure}

Notice that since $p_i \cdot \nu = 0$, if $x\in \partial(x_i + T Q_T)$ we have
\begin{align*}
u_{T}(x -  x_i)=  (\rho*u_{0,\nu})(x  - x_i) = (\rho*u_{0,\nu})(x  +p_i - x_i).
\end{align*}
Thus $u_S \in H^1(SQ_T; \R^d)$ and, if $x\in\partial \left(SQ_T \right)$ then
$u_S(x)=(\rho\ast u_{0,\nu})(x)$, so $u_S$ is admissible for the infimum in the definition of $g(S)$.
In particular,
\begin{align}
\notag
g(S) &\leq \frac{1}{S^{N-1}} \int_{SQ_T} \bigg[ W(x, u_S(x)) + |\nabla u_S(x)|^2 \bigg] dx \\ \notag
&=\frac{1}{S^{N-1}}\f_1(u_S, SQ_T)\\
\label{eq:4parts}
& =: I_1(T,S) + I_2(T,S,m) + I_3(T,S,m),
\end{align}
where
\begin{align*}
I_1(T,S) :=  & \frac{1}{S^{N-1}} \sum_{i=1}^{M_{T,S}} \f_1( u_S, x_i+TQ_T),\\
I_2(T,S,m) := & \frac{1}{S^{N-1}} \sum_{i=1}^{M_{T,S}} \f_1\left( u_S, \left(x_i +\left(T+\frac{1}{m}\right)Q_T\right) \setminus (x_i + TQ_T)  \right),\\
I_3(T,S,m) := & \frac{1}{S^{N-1}} \f_1( u_S, E_{T,S,m}),
\end{align*}
and we set
\[
E_{T,S,m} := S Q_T \setminus \bigcup_{i=1}^{M_{T,S}} \left(x_i + \left(T+\frac{1}{m}\right) Q_T\right).
\]

% It is worth pointing out the following properties of $\rho_L * u_{0, \nu}$ for $L > 0$. We will demonstrate that 
%\begin{equation}\label{eq:support}
%\nabla (\rho_L * u_{0, \nu})(x) = 0 \textrm{ if } |x \cdot \nu| \geq \frac{1}{L}
%\end{equation}
%and that 
%\begin{equation}\label{eq:strip}
%\| \nabla (\rho_L * u_{0, \nu})\|_{\infty} \leq CL.
%\end{equation}
%To prove these, we note that $u_{0, \nu}$ is a jump function and hence its distributional derivative is the vector measure $(b-a) \otimes \nu \hno \restr \nu^{\perp}$. Then we see
%\[\nabla( \rho_L * u_{0, \nu})(x) = \int_{B\left(x, \frac{1}{L}\right) \cap \nu^{\perp} } \rho_{L}(y) (b-a) \otimes \nu d \hno(y).\]
%Thus, if $| x \cdot \nu| \geq \frac{1}{L}$, we have $B\left(x, \frac{1}{L} \right) \cap \nu^{\perp} = \emptyset$ and thus $\nabla( \rho_L * u_{0, \nu})(x) = 0$. To see \eqref{eq:strip}, we can estimate
%\[ \hno\left(B\left(x, \frac{1}{L}\right) \cap \nu^{\perp}\right) \leq C \frac{1}{L^{N-1}}.\]
%On the other hand, since $\|\nabla \rho_L\|_{\infty} \leq L^N$, we have for every $x$ that
%\[|\nabla( \rho_L * u_{0, \nu})(x)| \leq C L^N \hno\left(B\left(x, \frac{1}{L}\right) \cap \nu^{\perp}\right) \leq C L\]
%and thus $\| \nabla (\rho_L * u_{0, \nu})\|_{\infty} \leq CL.$

We now bound each of terms $I_1, I_2, I_3$ separately. 
We start with $I_1(T,S)$. Since $x_i \in \mathbb{Z}^N$, the periodicity of $W$ together with \eqref{eq:ut} yield
\begin{align}\label{eq:I1}
\notag
I_1(T,S) &= \frac{1}{S^{N-1}} M_{S,T} \int_{T Q_T} \left[ W(x, u_T(x)) + |\nabla u_T(x)|^2 \right] dx \nonumber \\
&\leq \frac{1}{S^{N-1}} M_{T,S} T^{N-1} \bigg( g(T) + \frac{1}{T} \bigg).
\end{align}
In order to estimate $I_2(T,S,m)$, notice that, since for every $x\in\R^N$ the function
$t\mapsto (\rho * u_{0,\nu})(x+t\nu)$ is constant outside of an interval of size $1$, we have that for every $i\in\{1,\dots, M_{T,S}\}$ it holds
\begin{align}\label{eq:estrhotnu}
&\int_{\left(x_i +\left(T+\frac{1}{m}\right)Q_T\right) \setminus (x_i + TQ_T)}
   |\nabla(\rho_{1} * u_{0,\nu})(x+p_i-x_i) |^2 \dx \nonumber\\
&\hspace{4cm}\leq {\| \nabla(\rho * u_{0,\nu})\|}_{L^\infty}^2
    \bigg[ \bigg( T+\frac{1}{m}\bigg)^{N-1} - T^{N-1}\bigg]\,.
\end{align}
Thus, using \eqref{eq:estphimi} and \eqref{eq:estrhotnu} we obtain
\begin{align}\label{eq:I2}
I_2(T,S,m) &\leq \frac{C}{S^{N-1}} M_{T,S} \bigg[ \bigg( \frac{\sqrt{N}}{2} + 1 \bigg)
	(1+ {\|\nabla \varphi_{m,i}\|}_{L^\infty}^2)
    +  {\| \nabla(\rho* u_{0,\nu})\|}_{L^\infty}^2 \nonumber \\
&\hspace{2.5cm}  +   {\| \nabla(\rho* u_{0,\nu})\|}_{L^\infty}^2 \bigg]
    \bigg[ \bigg( T+\frac{1}{m}\bigg)^{N-1} - T^{N-1}\bigg] \nonumber \\
&\hspace{0.4cm}\leq \frac{C}{S^{N-1}} M_{T,S} \left( 1+m^2  \right)
    \bigg[ \bigg(T+\frac{1}{m}\bigg)^{N-1} - T^{N-1} \bigg]  \nonumber \\
&\hspace{0.4cm}\leq C \frac{T^{N-1}}{S^{N-1}}M_{T,S} \left(1+m^2 \right)
    \bigg[ \bigg(1 + \frac{1}{Tm} \bigg)^{N-1} - 1 \bigg] \nonumber \\
&\hspace{0.4cm}\leq C \frac{T^{N-1}}{S^{N-1}}M_{T,S}\left( 1+m^2 \right) 
    \left( \frac{N-1}{Tm} \right) =:J_2(T,S,m) 
\end{align}
where in the last step we used the inequality
\begin{equation}\label{eq:taylor}
(1+t)^{N-1}\leq 1+C(N-1)t
\end{equation}
for $t\ll 1$, that is valid here when $T\gg1$.

We can finally estimate $I_3(T,S,m)$ as
\begin{align}\label{eq:I3}
I_3(T,S,m) &=\frac{1}{S^{N-1}} \int_{E_{T,S,m}} \left[\, W(x, \rho * u_{0,\nu})
	+ |\nabla(\rho * u_{0,\nu})|^2 \,\right] \dx \nonumber \\
&\leq \frac{C}{S^{N-1}} \bigg| E_{T,S,m} \cap \bigg\{|x \cdot \nu|< 1 \bigg\}
    \bigg| \bigg( 1+ {\| \nabla(\rho* u_{0,\nu})\|}_{L^\infty}^2 \bigg) \nonumber \\
& \leq \frac{C}{S^{N-1}} \bigg[ S^{N-1} - M_{T,S} T^{N-1} \bigg] =:J_3(T,S,m) .
\end{align} 

Taking into account \eqref{eq:I2} and \eqref{eq:I3} we obtain
\begin{equation}\label{eq:limlimlim}
\lim_{m \to \infty} \lim_{T \to \infty} \lim_{S \to \infty} \,[\, J_2(T,S,m) + J_3(T,S,m) \,] = 0.
\end{equation}
Thus, in view of \eqref{eq:4parts}, \eqref{eq:I1}, \eqref{eq:mst} and \eqref{eq:ut}, we conclude \eqref{eq:estSTsigma} with
\begin{equation}\label{eq:Rmst}
R(m, S, T) := J_2(T,S,m) + J_3(T,S,m) .
\end{equation}
Notice that $R(m,S,T)$ does not depend on $\nu$ nor on $Q_T$.

Finally, to prove that $\sigma(\nu)<\infty$ for all $\nu\in\S^{N-1}$ we notice that, by sending $S\to\infty$ in \eqref{eq:estSTsigma}
we get
\[
\sigma(\nu)\leq g(T)+\lim_{S\to\infty}R(m,S,T).
\]
Since $g(T)<\infty$ and, by \eqref{eq:limlimlim} and \eqref{eq:Rmst}, $\lim_{S\to\infty}R(m,S,T)<\infty$ for all $T>0$, we conclude. 
\end{proof}

\begin{remark}\label{rem:exlimfixcube}
The proof of Lemma \ref{lem:estimatesigma} shows, in particular, that
\[
\lim_{T\to\infty} \frac{1}{T^{N-1}}\inf\Bigl\{\, \int_{T Q}\left[ W(y,u(y))+|\nabla u|^2  \right]\dd y \,:\,
    u\in \mathcal{C}(Q, T)   \,\Bigr\}
\]
exists, for every $\nu\in\S^{N-1}$ and every $Q\in\mathcal{Q}_\nu$. This will be used later in the proof of Lemma \ref{lem:sigmaQ}.
\end{remark}

Next we show that the definition of $\sigma(\nu)$ does not depend on the choice of the mollifier $\rho$ we choose to impose the boundary conditions.

\begin{lemma}\label{lem:indepmoll}
For every $\nu\in\S^{N-1}$ the definition of $\sigma(\nu)$ does not depend on the choice of the mollifier $\rho$.
\end{lemma}

\begin{proof}
Fix $\nu\in\S^{N-1}$ and let $\{T_n\}_{n\in\N}$ be such that $T_n\to\infty$ as $n\to\infty$.
Let $\rho^{(1)}, \rho^{(2)}\in C^\infty_c(B(0,1))$ be two mollifiers and let us denote by $\sigma(\nu,\rho^{(1)})$ and
$\sigma(\nu,\rho^{(2)})$ the functions defined as in Definition \ref{def:sigma} using $\rho^{(1)}$ and $\rho^{(2)}$, respectively, to impose the boundary conditions for the admissible class of functions.
Let $\{Q_n\}_{n\in\N}\subset \mathcal{Q}_\nu$ and $\{u^{(1)}_n\}_{n\in\N}\subset H^1(T_n Q_n;\R^d)$ with $u^{(1)}_n:=\rho^{(1)}\ast u_{0,\nu}$ on $\partial T_nQ_n$ be such that
\begin{equation}\label{eq:indrho1}
\lim_{n\to\infty} \frac{1}{T^{N-1}_n}\f_1(u_n^{(1)},T_nQ_n)=\sigma(\nu,\rho^{(1)}).
\end{equation}
For every $n\in\N$, consider the cubes $\left(T_n+1\right)Q_n$ and a function
$\varphi_n\in C^\infty((T_n +1)Q_n)$ with $0\leq \varphi_n,\leq 1$ such that $\varphi_{n}\equiv 1$ in $T_nQ_n$, $\varphi_{n}\equiv 0$ on $\partial[(T_n+1)Q_n]$.
For every $n\in\N$ define $u_n^{(2)}\in H^1((T_n+1)Q_n;\R^d)$ as
\[
u_n^{(2)}(x):=\left\{
\begin{array}{ll}
u_n^{(1)}(x) & \text{ if } x\in T_nQ_n,\\
\varphi_n(x)(\rho^{(1)}\ast u_{0,\nu})(x)
    +\left(1-\varphi_n(x)\right)(\rho^{(2)}\ast u_{0,\nu})(x) & \text{ otherwise }.
\end{array}
\right.
\]

Then $u_n^{(2)}=\rho^{(2)}\ast u_{0,\nu}$ on $\partial[(T+1)Q_n]$
and $u_n^{(2)}$ is constant (taking values $a$ or $b$) outside of the set
$\{(x',x_\nu)\in\R^N \,:\, x'\in Q'_n, x_\nu=s\nu \text{ for } s\in [-1,1] \}$ where
$Q'_n:=[(T_n+1)Q_n\setminus T_n Q_n]\cap\nu^\perp$.
We have
\begin{equation}\label{eq:indrho2}
\frac{1}{(T_n+1)^{N-1}}\f_1\left(u_n^{(2)}, \left(T_n+1\right)Q_n\right) \leq \frac{1}{T^{N-1}_n}\f_1(u_n^{(1)},T_nQ_n) + R_n,
\end{equation}
where 
\begin{align*}
R_n&:=\frac{1}{T_n^{N-1}}\int_{(T_n+1Q\setminus T_nQ_n} [\, W(y,u_n^{(2)}(y)) + |\nabla u_n^{(2)}(y)|^2 \,] \dy \\
& \leq \frac{C}{T_n^{N-1}}\left[ \left( T_n+1\right)^{N-1} - T_n^{N-1} \right] \\
&\leq \frac{C}{T_n}\,,
\end{align*}
where in the last inequality we use \eqref{eq:taylor}. Using \eqref{eq:indrho1} and \eqref{eq:indrho2} we get
\[
\sigma(\nu,\rho^{(2)})\leq \liminf_{n\to\infty }\frac{1}{(T_n+1)^{N-1}}
    \f_1\left(u_n^{(2)}, \left(T_n+1\right)Q_n\right) \leq \sigma(\nu,\rho^{(1)}).
\]
By swapping the roles of $\rho^{(1)}$ and $\rho^{(2)}$ we get the desired result.
\end{proof}

We now prove a regularity property for the function $\sigma$.

\begin{proposition}\label{prop:sigma}
The function $\sigma:\S^{N-1}\to[0,\infty)$ is upper semi-continuous.
\end{proposition}

\begin{proof}
\emph{Step 1.}
Fix $\nu\in\S^{N-1}$ and let $\{\nu_n\}_{n\in\N}\subset\S^{N-1}$ be such that $\nu_n\to\nu$ as $n\to\infty$.
We first prove that, for fixed $T>0$, the function $\nu\mapsto g(\nu,T)$ is continuous.
We claim that $\limsup_{n\to\infty}g(\nu_n,T)\leq g(\nu,T)$.
Fix $\varepsilon>0$. Let $Q_\nu\in\mathcal{Q}_\nu$ and $u\in \mathcal{C}(TQ_\nu,\nu)$ be such that
\begin{equation}\label{eq:approxnu}
\left|T^{N-1}g(\nu,T)-\int_{T Q_\nu}\left[\, W(y,u(y))+|\nabla u|^2  \,\right]\dd y\right|<\e\,.
\end{equation}
Without loss of generality, by density, we can assume that $u\in L^\infty(\o;\R^d)$.
For every $n\in\N$, let $\mathcal{R}_n:\R^N\to\R^N$ be a rotation such that $\mathcal{R}_n\nu_n=\nu$ and $\mathcal{R}_n\to\mathrm{Id}$ as $n\to\infty$, where $\mathrm{Id}:\R^N\to\R^N$ is the identity map.
Moreover, thanks to Lemma \ref{lem:indepmoll}, we can assume the mollifier $\rho$ and the rotations $\mathcal{R}_n$ to be such that $\rho(\mathcal{R}_ny) = \rho(y)$ for all $y\in\R^N$ and $n\in\N$.
Notice that it is possible to satisfy this condition because  $\mathcal{R}_n\nu_n=\nu$ for all $n\in\N$.
Define $u_n\in\mathcal{C}(TQ_{\nu_n},\nu_n)$ as $u_n(y):=u(\mathcal{R}_n y)$.
By \eqref{eq:approxnu} we have
\begin{align}\label{eq:ulsgT}
T^{N-1}g(\nu_n, T)&\leq 
\int_{T Q_{\nu_n}}\left[ W(y,u_n(y))+|\nabla u_n|^2  \right]\dd y \nonumber \\
&\leq \int_{T Q_{\nu}}\left[ W(y,u(y))+|\nabla u|^2  \right]\dd y + \delta_n \nonumber \\
&\leq T^{N-1}g(\nu,T)+\e+\delta_n\,,
\end{align}
where
\[
\delta_n:=\left|\,  \int_{T Q_{\nu_n}} W(y,u_n(y)) \dd y - \int_{T Q_{\nu}}W(y,u(y))  \dd y  \,\right|\,.
\]
We claim that $\delta_n\to0$ as $n\to\infty$. Since $\e>0$ is arbitrary in \eqref{eq:ulsgT}, this would confirm the claim.

Fix $\eta>0$ and let $M:=C(1+\|u\|_{L^\infty}^q)$, where $C>0$ and $q\geq2$ are given by (H4).
Let $K\subset\R^N$ be a compact set such that $TQ_\nu\subset K$ and $TQ_{\nu_n}\subset K$ for every $n\in\N$.
Notice that $W(x,u(x))\leq M$ for all $x\in TQ_\nu$.
Using the Scorza-Dragoni theorem (see \cite[Theorem 6.35]{FonLeo}) and the Tietze extension theorem (see \cite[Theorem A.5]{FonLeo}), we can find a compact set $E\subset K$ with $|E|<\eta$ and continuous map
$\widetilde{W}:K\times\R^d\to[0,\infty)$ such that $\widetilde{W}(x,\cdot)=W(x,\cdot)$ for all
$x\in K\setminus E$ and $|\widetilde{W}(x,u(x))|\leq M$ for every $x\in K$.
We claim that
\begin{equation}\label{eq:estw1}
\int_{TQ_\nu} \left|\, W(y,u(y))-\widetilde{W}(y,u(y)) \,\right| \dy\leq C\eta\,,
\end{equation}
and that
\begin{equation}\label{eq:estw2}
\int_{TQ_{\nu_n}} \left|\, W(y,u_n(y))-\widetilde{W}(y,u_n(y))  \,\right| \dy\leq C\eta\,.
\end{equation}
Indeed
\begin{align*}
\int_{TQ_\nu} \left|\, W(y,u(y))-\widetilde{W}(y,u(y)) \,\right| \dy&=
    \int_{E} \left|\, W(y,u(y))-\widetilde{W}(y,u(y)) \,\right| \dy\\
&\leq 2M |E|\\
&\leq 2M \eta\,.
\end{align*}
A similar argument yields \eqref{eq:estw2}.
Since $TQ_\nu$ is bounded
\begin{equation}\label{eq:estw3}
\int_{T Q_\nu} \left|\, \widetilde{W}(\mathcal{R}_n y,u(y)) - \widetilde{W}(y,u(y)) \,\right| \dy\to0\,,
\end{equation}
as $n\to\infty$. Thus, from \eqref{eq:estw1}, \eqref{eq:estw2} and \eqref{eq:estw3} we obtain
\[
\limsup_{n\to\infty}\delta_n\leq 2C\eta\,.
\]
The claim follows from the arbitrariness of $\eta$.

In an analogous way it is possible to show that $\liminf_{n\to\infty}g(\nu_n,T)\geq g(\nu,T)$, and thus we conclude that the function $\nu\to g(\nu,T)$ is continuous.\\

\emph{Step 2.} Fix $\nu \in \S^{N-1}$, $\e>0$, and let $T>0$ be such that
\begin{equation}\label{eq:epssigmanu}
|g(\nu,T)-\sigma(\nu)|<\e\,.
\end{equation}
Let $\{\nu_n\}_{n\in\N}$ be a sequence converging to $\nu$. By Step 1 we have that
\begin{equation}\label{eq:nunconv}
\lim_{n \rga \infty} g(\nu_n,T) = g(\nu,T).
\end{equation}
Then, for $S>T+3+\sqrt{N}$, using \eqref{eq:estSTsigma} and \eqref{eq:epssigmanu} we get, for $m\in\{1,\dots,T\}$,
\begin{align*}
g( \nu_n, S) &\leq g(\nu_n, T)+R(m,S,T) \nonumber \\
&= g(\nu,T) + g(\nu_n, T) - g(\nu, T) + R(m,S,T) \nonumber \\
&\leq \sigma(\nu) + \e + g(\nu_n,T) - g(\nu,T) + R(m,S,T).
\end{align*}
Taking the limit as $S \to\infty$ we obtain
\[
\sigma(\nu_n) \leq \sigma(\nu) + \e + g(\nu_n,T) - g(\nu,T) + \lim_{S \rga \infty} R(m,S,T).
\]
Letting $n \rga \infty$, by \eqref{eq:nunconv}
\[
\limsup_{n \rga \infty} \sigma(\nu_n) \leq \sigma(\nu) + \e + \lim_{S \rga \infty} R(m,S,T).
\]
Finally, taking $T \to \infty$ and then $m\to\infty$, using \eqref{eq:Rmst}, we conclude that
\[
\limsup_{n \rga \infty} \sigma(\nu_n) \leq \sigma(\nu) + \e
\]
for every $\e > 0$, and thus we obtain upper-semicontinuity.
\end{proof}

The following technical results, that will be fundamental in the proof of the limsup inequality (see Proposition \ref{prop:limsup}), aim at providing two different ways to obtain,
for  $\nu\in\S^{N-1}$, the value $\sigma(\nu)$.

\begin{lemma}\label{lem:goodcubes}
Let $\nu\in\Lambda$. Then
\begin{equation}\label{eq:sigmaandLambda}
\sigma(\nu)=\lim_{T\to\infty} g^\Lambda(\nu,T)\,,
\end{equation}
where
\[
g^\Lambda(\nu,T):=\frac{1}{T^{N-1}}\inf\Bigl\{\, \int_{T Q_\nu}\left[ W(y,u(y))+|\nabla  u|^2  \right]\dd y \,:\,
    Q_\nu\in\mathcal{Q}_\nu^\Lambda,\, u\in \mathcal{C}(Q_\nu, T)   \,\Bigr\}\,,
\]
and $\mathcal{Q}_\nu^\Lambda$ is the family of cubes with unit length side centered at the origin with two faces orthogonal to $\nu$ and the other faces orthogonal to elements of $\Lambda$.
\end{lemma}

\begin{proof}
Fix $\nu\in\Lambda$. From the definition of $\sigma(\nu)$ it follows that
\begin{equation}\label{eq:estbelow}
\sigma(\nu)\leq\liminf_{T\to\infty} g^\Lambda(\nu,T)\,.
\end{equation}
Let $\{T_n\}_{n\in\N}$ with $T_n\to\infty$ as $n\to\infty$. By Lemma \ref{lem:estimatesigma}, let $\{Q_n\}_{n\in\N}\subset\mathcal{Q}_\nu$ and $\{u_n\}_{n\in\N}$ with $u_n\in\mathcal{C}(Q_n, T_n)\cap L^\infty(T_nQ_n;\R^d)$ be such that
\begin{equation}\label{eq:withsigmanu}
\lim_{n\to\infty} \frac{1}{T_n^{N-1}}\f_1(u_n,T_nQ_n)=\sigma(\nu).
\end{equation}
For every fixed $T_n$, an argument similar to the one used in Step 1 of the proof of Proposition \ref{prop:sigma} together with Lemma \ref{lem:denserotations} ensure that it is possible to find rotations $\mathcal{R}_n:\R^N\to\R^N$ with $\mathcal{R}_n(e_N)=\nu$ and $\mathcal{R}_n(e_i)\in\Lambda$ for all $i=1,\dots,N-1$ such that
\begin{equation}\label{eq:estrotcubes}
| \f_1(u_n,T_nQ_n)-\f_1(\widetilde{u}_n,T_n \mathcal{R}_n(Q_n))|<\frac{1}{n},
\end{equation}
where $\widetilde{u}_n(x):=u_n(\mathcal{R}_n^{-1}x)$.
Thus
\begin{align}\label{eq:estabove}
\limsup_{n\to\infty} g^\Lambda(\nu, T) &\leq \limsup_{n\to\infty}\frac{1}{T^{N-1}_n}
    \f_1(\widetilde{u}_n,T_n \mathcal{R}_n(Q_n)) \nonumber \\
&\leq \limsup_{n\to\infty}\frac{1}{T^{N-1}_n}\f_1(u_n,T_n Q_n) \nonumber \\
&=\sigma(\nu),
\end{align}
where the last step follows from \eqref{eq:withsigmanu}, while in the second to last step we used \eqref{eq:estrotcubes}.
By \eqref{eq:estbelow} and \eqref{eq:estabove} and the arbitrariness of the sequence $\{T_n\}_{n\in\N}$, we conclude \eqref{eq:sigmaandLambda}.
\end{proof}

\begin{lemma}\label{lem:sigmaQ}
For $\nu\in\S^{N-1}$ and $Q\in\mathcal{Q}_\nu$ define
\[
\sigma^{Q}(\nu):=\lim_{T\to\infty} g^{Q}(\nu,T)\,,
\]
where
\[
g^{Q}(\nu,T):=\frac{1}{T^{N-1}}\inf\Bigl\{\, \int_{T Q}\left[ W(y,u(y))+|\nabla  u|^2  \right]\dd y \,:\,
    u\in \mathcal{C}(Q, T)   \,\Bigr\}\,.
\]
Then there exists $\{Q_n\}_{n\in\N}\subset\mathcal{Q}_\nu$ such that $\sigma^{Q_n}(\nu)\to \sigma(\nu)$ as $n\to\infty$.
In particular, if $\nu\in\Lambda$ it is possible to take $\{Q_n\}_{n\in\N}\subset \mathcal{Q}_\nu^\Lambda$.
\end{lemma}

\begin{proof}
First of all notice that, in view of Remark \ref{rem:exlimfixcube}, $\sigma^{Q}(\nu)$ is well defined.
By definition, we have $\sigma(\nu)\leq \sigma^{Q}(\nu)$ for all $Q\in\mathcal{Q}_\nu$.
Thus, it suffices to prove that it is possible to find a sequence $\{Q_n\}_{n\in\N}\subset\mathcal{Q}_\nu$ such that
$\sigma^{Q_n}(\nu)\leq\sigma(\nu)+R_n$, where $R_n\to0$ as $n\to\infty$.
Let $\{T_n\}_{n\in\N}$ be an increasing sequence with $T_n\to\infty$ as $n\to\infty$ such that
\[
g(\nu,T_n)\leq \sigma(\nu)+\frac{1}{n}.
\]
It is then possible to find $\{Q_n\}_{n\in\N}\subset\mathcal{Q}_\nu$ (or, using Lemma \ref{lem:goodcubes}, $\{Q_n\}_{n\in\N}\subset\mathcal{Q}^\Lambda_\nu$ in case $\nu\in\Lambda$) such that for all $n\in\N$ it holds
\begin{equation}\label{eq:estsigmaQ1}
g^{Q_n}(\nu,T_n)\leq g(\nu,T_n)+\frac{1}{n}.
\end{equation}
An argument similar to the one used in Lemma \ref{lem:estimatesigma} to establish \eqref{eq:estSTsigma} shows that
for every $\nu\in\S^{N-1}$, $Q\in\mathcal{Q}_\nu$, $T>0$, $S>T+3+\sqrt{N}$ and $m\in\{1,\dots,T\}$, it holds
\begin{equation}\label{eq:estsigmaQ2}
g^{Q}(\nu,S)\leq g^{Q}(\nu,T)+R(m,S,T), 
\end{equation}
where $R(m,S,T)$ is independent of $\nu\in\S^{N-1}$ and of $Q\in\mathcal{Q}_\nu$ (see \eqref{eq:Rmst}), and is such that
\[
\lim_{m\to\infty}\lim_{T\to\infty}\lim_{S\to\infty}R(m,S,T)=0.
\]
In particular, for all $n\in\N$, it is possible to choose $m_n\in\{1,\dots,T_n\}$ such that 
\begin{equation}\label{eq:estsigmaQ3}
\lim_{n\to\infty}\lim_{S\to\infty}R(m_n,S,T_n)=0.
\end{equation}
Thus, we get
\begin{equation}\label{eq:estsigmaQ4}
g^{Q_n}(\nu, S)\leq g^{Q_n}(\nu,T_n)+R(m_n,S,T_n).
\end{equation}
From \eqref{eq:estsigmaQ1} and \eqref{eq:estsigmaQ4}, sending $S\to\infty$, we get
\[
\sigma^{Q_n}(\nu)\leq \sigma(\nu)+\frac{2}{n}+\lim_{S\to\infty}R(m_n,S,T_n).
\]
Using \eqref{eq:estsigmaQ3} we conclude that
\[
\sigma(\nu)=\lim_{n\to\infty}\sigma^{Q_n}(\nu)\,.
\]
\end{proof}

%%%%%%%%%%%%%%%%%%%%%%%%%%%%%%%%%%%%%%%%%%%%%%%%
%%%%%%%%%%%%%%%%%%%%%%%%%%%%%%%%%%%%%%%%%%%%%%%%
%%%%%%%%%%%%%%%%%%%%%%%%%%%%%%%%%%%%%%%%%%%%%%%%
%%%%%%%%%%%%%%%%%%%%%%%%%%%%%%%%%%%%%%%%%%%%%%%%

\section{Compactness}

\begin{proposition}\label{prop:comp}
Let $\{u_n\}_{n\in\N}\subset \sp$ be a sequence with $\sup_{n\in\N}\f_{\e_n}(u_n)<+\infty$, where $\e_n\to0^+$.
Then there exists $u\in BV(\o;\{a,b\})$ such that, up to a subsequence (not relabeled), $u_n\to u$ in $L^1(\o;\R^d)$.
\end{proposition}

\begin{proof}
Let $\widetilde{W}:\R^d\to[0,\infty)$ be the continuous function given by (H3).
Let $R>0$ be such that $\frac{1}{C}|p|^q-C>0$ for $|p|>R$, where $C>0$ and $q\geq2$ are as in (H4), and $|a|, |b|<R$.
Take a function $\varphi\in C^\infty(\R^d)$ such that $\varphi\equiv1$ in $B_R(0)$ and $\varphi\equiv0$ in $B_{2R}(0)$.
Define the function $\mathcal{W}:\R^d\to[0,\infty)$ by
\[
\mathcal{W}(p):=\varphi(p)\widetilde{W}(p)+(1-\varphi(p))\left( \frac{1}{C}|p|^q-C \right),
\]
for $p\in\R^d$. Notice that $\mathcal{W}(p)=0$ if and only if $p\in\{a,b\}$.
Since $\widetilde{W}(p)\leq W(x,p)$ for a.e. $x\in Q$, we get
\[
\f_{\e_n}(u_n)\geq \int_\o \left[ \, \frac{1}{\e} \mathcal{W}(u_n(x)) + \e |\nabla u_n(x)|^2  \,\right] \dx =: \widetilde{\f}_{\e_n}(u_n)\,,
\]
and, in turn, we have that $\sup_{n\in\N}\widetilde{\f}_{\e_n}(u_n)<+\infty$. We now proceed as in \cite{fonseca89} to obtain a subsequence of $\{u_n\}_{n\in\N}$ and $u\in BV(\o;\{a,b\})$ such that $u_n\to u$ in $L^1(\o;\R^d)$.
\end{proof}

%%%%%%%%%%%%%%%%%%%%%%%%%%%%%%%%%%%%%%%%%%%%%%%%
%%%%%%%%%%%%%%%%%%%%%%%%%%%%%%%%%%%%%%%%%%%%%%%%
%%%%%%%%%%%%%%%%%%%%%%%%%%%%%%%%%%%%%%%%%%%%%%%%
%%%%%%%%%%%%%%%%%%%%%%%%%%%%%%%%%%%%%%%%%%%%%%%%

\section{Liminf inequality}

This section is devoted to the proof of the liminf inequality.

\begin{proposition}\label{prop:liminf}
Given a sequence $\{\e_n\}_{n\in\N}$ with $\e_n\to0^+$, let $\{u_n\}_{n\in\N}\subset \sp$ be such that $u_n\to u$ in $L^1(\o;\R^d)$. Then
\[
\f_0(u)\leq\liminf_{n\to\infty}\f_{\e_n}(u_n)\,.
\]
\end{proposition}

\begin{proof}
Let $\{u_n\}_{n\in\N}\subset H^1(\o;\R^d)$ with $u_n\to u$ in $L^1(\o;\R^d)$.
Without loss of generality, and possibly up to a subsequence, we can assume that
\begin{equation}\label{eq:boundfun}
\liminf_{n\to\infty}\f_{\e_n}(u_n)=\lim_{n\to\infty}\f_{\e_n}(u_n)<\infty\,.
\end{equation}
By Proposition \ref{prop:comp}, we get $u\in BV(\o;\{a,b\})$. Set $A:=\{u=a\}$.
Consider, for every $n\in\N$, the finite nonnegative Radon measure
\[
\lambda_n:=\left[\, \frac{1}{\e} W\left(\frac{x}{\e},u_n(x)\right) + \e |\nabla  u_n(x)|^2\,\right]\mathcal{L}^N\restr\o\,.
\]
From \eqref{eq:boundfun} we have that $\sup_{n\in\N}\lambda_n(\o)<\infty$. Thus, up to a subsequence (not relabeled), $\lambda_n\stackrel{w^*}{\rightharpoonup}\lambda$, for some finite nonnegative Radon measure $\lambda$ in $\o$. In particular,
\begin{equation}\label{eq:ineq1}
\liminf_{n\to\infty} \f_{\e_n}(u_n)= \liminf_{n\to\infty} \lambda_n(\o)\geq \lambda(\o)\,.
\end{equation}
We claim that for $\hno$-a.e. $x_0\in \partial^* A$ it holds
\begin{equation}\label{eq:ineq2}
\frac{\dd \lambda}{\dd\mu}(x_0)\geq\sigma(\nu_A(x_0))\,,
\end{equation}
where $\mu:=\hno\restr\partial^* A$.
The liminf inequality follows from \eqref{eq:ineq1} and \eqref{eq:ineq2}.
The rest of the proof is devoted at showing the validity of \eqref{eq:ineq2}.\\

\emph{Step 1.}  For $\hno$-a.e. $x\in\partial^* A$ we have
\begin{equation}\label{eq:finitedensity}
\frac{\dd \lambda}{\dd\mu}(x)<\infty\,.
\end{equation}
Fix $x_0\in\partial^* A$ satisfying \eqref{eq:finitedensity} and a cube $Q_\nu\in\mathcal{Q}_{\nu}$, with $\nu:=\nu_A(x_0)$.
Let $\{\delta_k\}_{k\in\N}$ be a sequence with $\delta_k\to0$ as $k\to\infty$, such that $\lambda(\partial Q_\nu(x_0,\delta_k))=0$,
where $Q_\nu(x_0,\delta_k):=x_0+\delta_k Q_\nu$ for all $k\in\N$. Then it holds
\begin{equation}\label{eq:limdensity}
\frac{\dd \lambda}{\dd\mu}(x_0)
=\lim_{k\to\infty}\frac{\lambda (Q_\nu(x_0,\delta_k))}{\delta_k^{N-1}}
=\lim_{k\to\infty}\lim_{n\to\infty} \frac{\lambda_n (Q_\nu(x_0,\delta_k))}{\delta_k^{N-1}}\,.
\end{equation}
We have
\begin{align}\label{eq:firstpartcompdensity}
&\frac{\lambda_n (Q_{\nu}(x_0,\delta_k))}{\delta_k^{N-1}}=
    \frac{1}{\delta_k^{N-1}}\int_{Q_{\nu}(x_0,\delta_k)} \left[ \, \frac{1}{\e_n} W\left(\frac{x}{\e_n},u_n(x)\right)
        + \e_n |\nabla  u_n(x)|^2  \,\right] \dx  \nonumber\\
&\hspace{0.6cm}=\delta_k \int_{Q_{\nu}} \left[ \, \frac{1}{\e_n}
    W\left(\frac{x_0+\delta_k z}{\e_n},u_n(x_0+\delta_k z)\right)+ \e_n |\nabla  u_n(x_0+\delta_k z)|^2  \,\right] \dd z \nonumber\\
&\hspace{0.6cm}=\int_{Q_{\nu}-\frac{\e_n}{\delta_k}s_n} \left[ \, \frac{\delta_k}{\e_n}
    W\left(\frac{\delta_k}{\e_n}y, u_{n}(y^n_k) \right)+ \e_n |\nabla  u_n(y^n_k)|^2  \,\right] \dd y\,,
\end{align}
where in the last step, for the sake of simplicity, we set $y^n_k:=x_0+\delta_k y+ \e_n s_n$, we wrote $\frac{x_0}{\e_n}=m_n-s_n$, with $m_n\in\Z^N$ and $|s_n|\leq\sqrt{N}$, and we used the periodicity of $W$ to simplify, for $z=y+\frac{\e_n}{\delta_k}s_n$, $z\in Q_\nu$,
\[
W\left(\frac{x_0+\delta_k z}{\e_n},\cdot\right)
=W\left(\frac{x_0+\delta_k (y+\frac{\e_n}{\delta_k}s_n)}{\e_n}, \cdot\right)
=W\left(m_n+\frac{\delta_k}{\e_n}y,\cdot\right)
=W\left(\frac{\delta_k}{\e_n}y,\cdot\right)\,.
\]
Consider the functions $u_{k,n}(x):=u_n(x_0+\delta_k x)$, for $n,k\in\N$.
We claim that
\begin{equation}\label{eq:convlp}
\lim_{k\to\infty}\lim_{n\to\infty}\|u_{k,n} -u_{0,\nu_A(x_0)}\|_{L^1(Q_\nu;\R^d)}=0\,,
\end{equation}
where $u_{0,\nu_A(x_0)}$ is defined as in \eqref{eq:u0}.
Set $Q^+_\nu:=Q_\nu\cap\{ x\in\R^N\,:\, x\cdot\nu>0 \}$ and $Q^-_\nu$ its complement in $Q_\nu$. We get
\begin{align*}
\lim_{k\to\infty}&\lim_{n\to\infty}\|u_{k,n} -u_{0,\nu_A(x_0)}\|_{L^1(Q_\nu;\R^d)}\\
&\hspace{0.6cm}=\lim_{k\to\infty}\lim_{n\to\infty} \left[\, \int_{Q^-_\nu} |u_n(x_0+\delta_k x)-a|\dx
    +\int_{Q^+_\nu} |u_n(x_0+\delta_k x)-b|\dx \,\right]\\
&\hspace{0.6cm}=\lim_{k\to\infty} \left[\, \int_{Q^-_\nu} |u(x_0+\delta_k x)-a|\dx
    +\int_{Q^+_\nu} |u(x_0+\delta_k x)-b|\dx \,\right]\\
&\hspace{0.6cm}=\lim_{k\to\infty}\frac{1}{\delta_k^N}\left[\, \int_{Q_\nu(x_0,\delta_k)\cap H^-_\nu} |u(y)-a|\dy
      +\int_{Q_\nu(x_0,\delta_k)\cap H^+_\nu} |u(y)-b| \dy\,\right]\\
&\hspace{0.6cm}=|b-a| \lim_{k\to\infty}\left[\, \frac{|Q_\nu(x_0,\delta_k)\cap H^-_\nu\cap  B|}{\delta_k^N} 
    + \frac{|Q_\nu(x_0,\delta_k)\cap H^+_\nu\cap  A|}{\delta_k^N} \,\right]\\
&\hspace{0.6cm}=0\,,
\end{align*}
where  $H^+_\nu:=\{ x\in\R^N\,:\,x\cdot\nu>x_0\cdot\nu \}$, $H^-_\nu$ is its complement in $\R^N$ and $B:=\o\setminus A$. The last step follows from (i) of Theorem \ref{thm:DeGiorgi}.\\

\emph{Step 2.} Using a diagonal argument, and \eqref{eq:convlp}, it is possible to find an increasing sequence $\{n_k\}_{k\in\N}$ such that, setting
\[
\eta_k:=\frac{\e_{n_k}}{\delta_k}\,,\quad\quad\quad
x_k:= \eta_k s_{n_k}\,,\quad\quad\quad
w_k(x):=u_{k,n_k}\left(x-x_k\right),
\]
the following hold:
\begin{itemize}
\item[(i)] $\lim_{k\to\infty}\eta_k=0$;
\item[(ii)] $\lim_{k\to\infty} x_k=0$;
\item[(iii)] $w_k\to u_{0,\nu}$ in $L^q(Q_\nu;\R^d)$ for all $q\geq1$;
\item[(iv)] we have
\begin{align*}
\lim_{k\to\infty}\delta_k &\int_{Q_{\nu}-x_k} \left[ \,
    \frac{1}{\e_{n_k}} W\left(\frac{y}{\eta_k}, u_{n_k}(y^{n_k}_k) \right)+ \e_{n_k} |\nabla  u_{n_k}(y^{n_k}_k)|^2  \,\right] \dd y\\
&=\lim_{k\to\infty}\lim_{n\to\infty}\delta_k \int_{Q_{\nu}-\frac{\e_n}{\delta_k}s_n} \left[ \, \frac{1}{\e_n} 
    W\left(\frac{\delta_k}{\e_n}y,u_n(y^n_k) \right)+ \e_n |\nabla  u_n(y^n_k)|^2  \,\right] \dd y\,.
\end{align*}
\end{itemize}
From \eqref{eq:limdensity}, \eqref{eq:firstpartcompdensity} and (iv) we get
\begin{align*}
\frac{\dd \lambda}{\dd\mu}(x_0)
&= \lim_{k\to\infty}\int_{Q_{\nu}-x_k} \left[ \, \frac{1}{\eta_k}
    W\left(\frac{y}{\eta_k}, w_k(y) \right) + \eta_k |\nabla  w_k(y)|^2  \,\right] \dd y\,.
\end{align*}
Let $Q_k$ be the largest cube contained in $Q_{\nu}-x_k$ centered at zero and having the same principal axes of $Q_\nu$. Since $x_k\to0$ as $k\to\infty$, $Q_k\subset Q_\nu-x_k$ for $k$ large and the integrand is nonnegative, we have that
\begin{align}\label{eq:secondpartcompdensity}
\frac{\dd \lambda}{\dd\mu}(x_0)&\geq\limsup_{k\to\infty} \int_{Q_k} \left[ \,
    \frac{1}{\eta_k} W\left(\frac{y}{\eta_k}, w_k(y) \right)+ \eta_k |\nabla  w_k(y)|^2  \,\right] \dd y\,. \\ \nonumber
\end{align}

\emph{Step 3.} Finally we modify $w_k$ close to $\partial Q_k$ in order to render it an admissible function for the infimum problem defining $\sigma(\nu)$ as in Definition \ref{def:sigma}.
Using Lemma \ref{lem:DeGslicing} we find a sequence $\{\bar{w}_k\}_{k\in\N}\subset H^1(Q_\nu;\R^d)$ such that
\begin{equation}\label{eq:ineqvkwk}
\liminf_{k\to\infty}\f_{\eta_k}(w_k,Q_k)\geq\limsup_{k\to\infty}\f_{\eta}(\bar{w}_k, Q_\nu)\,,
\end{equation}
and with $\bar{w}_k=(\widetilde{u_k})_{1/\eta_k,\nu}$ on $\partial Q_\nu$, where $(\widetilde{u_k})_{1/\eta_k,\nu}$  is defined as in \eqref{eq:u0conv}. Hence, by \eqref{eq:secondpartcompdensity} and \eqref{eq:ineqvkwk}
\begin{align*}
\frac{\dd \lambda}{\dd\mu}(x_0)&\geq
\limsup_{k\to\infty}\int_{Q_\nu} \left[ \, \frac{1}{\eta_k} W\left(\frac{y}{\eta_k}, 
    \bar{w}_k(y) \right)+ \eta_k |\nabla  \bar{w}_k(y)|^2  \,\right] \dd y\\
&=\limsup_{k\to\infty}\int_{\frac{1}{\eta_k}Q_\nu} \left[ \, \eta_k^{N-1}
    W\left(z, \bar{w}_k(\eta_k z) \right)+ \eta_k^{N+1} |\nabla  \bar{w}_k(\eta_k z)|^2  \,\right] \dd z\\
&=\limsup_{k\to\infty}\,\eta_k^{N-1} \int_{\frac{1}{\eta_k}Q_\nu} \left[ \, 
    W\left(z, v_k(z) \right)+ |\nabla  v_k(z)|^2  \,\right] \dd z\\
&\geq \sigma(\nu)\,,
\end{align*}
since $\bar{w}_k\in \mathcal{C}(Q_\nu,\frac{1}{\eta_k})$, where $v_k(z):=\bar{w}_k(\eta_k z)$, and this concludes the proof.
\end{proof}

%%%%%%%%%%%%%%%%%%%%%%%%%%%%%%%%%%%%%%%%%%%%%%%%
%%%%%%%%%%%%%%%%%%%%%%%%%%%%%%%%%%%%%%%%%%%%%%%%
%%%%%%%%%%%%%%%%%%%%%%%%%%%%%%%%%%%%%%%%%%%%%%%%
%%%%%%%%%%%%%%%%%%%%%%%%%%%%%%%%%%%%%%%%%%%%%%%%

\section{Limsup inequality}

In this section we construct a recovery sequence.

\begin{proposition}\label{prop:limsup}
Let $u\in BV(\o; \{a,b\})$. Given a sequence $\{\e_n\}_{n\in\N}$ with $\e_n\to0^+$ as $n \to \infty$, there exist
 $\{u_n\}_{n \in \N} \subset \sp$ with $u_n \to u$ in $L^1(\o;\R^d)$ as $n \to \infty$ such that
\begin{equation}\label{eq:limsup}
\limsup_{n \to \infty} \f_{\e_n}(u_n)\leq\f_0(u)\,.
\end{equation}
\end{proposition}

\begin{proof}
Notice that it is enough to prove the following: given any sequence $\{\e_n\}_{n\in\N}$ with $\e_n\to0$ as $n\to\infty$, it is possible to extract a subsequence $\{\e_{n_k}\}_{k\in\N}$ for which there exists 
 $\{u_k\}_{k \in \N} \subset \sp$ with $u_k \to u$ in $L^1(\o;\R^d)$ as $k \to \infty$ such that
\[
\limsup_{k \to \infty} \f_{\e_{n_k}}(u_k)\leq\f_0(u)\,.
\]
Since $L^1(\o;\R^d)$ is separable, we conclude using the Urysohn property of the $\Gamma$-limit (see \cite[Proposition 8.3]{dalmaso93}).\\

\emph{Case 1.} Assume that the set $A:=\{u=a\}$ is a $\Lambda$-polyhedral set (see Definition \ref{def:poly}).
We need to localize the $\Gamma$-limit of our sequence of functionals.
For $\{\delta_n\}_{n\in\N}$ with $\delta_n\to0$, $v \in L^1(\o;\R^d)$ and $U\in\cA(\o)$ we set
\[
\W_{\{ \del_n \}}(v; U) := \inf \bigg\{ \liminf_{n \to \infty} \f_{\delta_n}(v_n,U) :
v_n \to v \textrm{ in } L^1(U; \R^d), v_n \in H^1(U; \R^d) \bigg\}\,.
\]
Let $\mathcal{C}$ be the family of all open cubes in $\o$ with faces parallel to the axes, centered at points $x \in \o \cap \Q$ and with rational edgelength. Denote by $\cR$ the countable subfamily of $\cA(\o)$ whose elements are $\o$ and all finite unions of elements of $\mathcal{C}$, \emph{i.e.},
\[
\cR := \{\o\}\, \cup \, \Bigg\{ \bigcup_{i=1}^k C_i : k \in \N, C_i \in \mathcal{C} \Bigg\}.
\]

Let $\e_n \to 0^+$. We will select a suitable subsequence in the following manner. We enumerate the elements of $\mathcal{R}$ by $\{R_i\}_{i \in \N}$. First considering $R_1$, by a diagonalization argument we can find a subsequence $\{\e_{n_j}\}_{j \in \N} \subset \{\e_n\}_{n \in N}$ and functions $\{u^{R_1}_{j}\}_{j \in N} \subset H^1(R_1; \R^d)$ such that
\[
u^{R_1}_j \to u \quad\textrm{ in }\, L^1(R_1;\R^d),
\]
and
\[
\W_{\{\e_{n_j}\}}(u; R_1) = \lim_{j \to \infty} \f_{\e_{n_j}}(u^{R_1}_j, R_1).
\]
Now, considering $R_2$, we can extract a further subsequence $\{ \e_{n_{j_k}} \}_{k \in \N}$ and find functions $\{u^{R_2}_k\}  \subset H^1(R_2; \R^d)$ such that
\[ u^{R_2}_{k} \to u \quad\textrm{ in }\, L^1(R_2;\R^d), \quad u^{R_1}_{j_k} \to u \quad\textrm{ in }\, L^1(R_1;\R^d), \]
and
\[
\W_{\left\{\e_{n_{j_k}}\right\}}(u; R_2) = \lim_{k \to \infty} \f_{\e_{n_{j_k}}}(u^{R_2}_k, R_2) ,\quad \W_{\left\{\e_{n_{j_k}}\right\}}(u; R_1) = \lim_{k \to \infty} \f_{\e_{n_{j_k}}}(u^{R_1}_{j_k}, R_1).
\]
Continuing along the $\{R_i\}$ in this fashion and employing a further diagonalization argument, we can assert the existence of a subsequence $\{\e_n^{\cR}\}_{n\in\N}$ of $\{\e_n\}_{n\in\N}$ with the following property:
for every $C \in \cR$, there exists a sequence $ \{ \uCR \}_{n\in\N} \subset H^1(C; \R^d)$ such that
\[
\uCR \to u \quad\textrm{ in }\, L^1(C;\R^d),
\]
and
\begin{equation}\label{eq:limUC}
\W_{\{\e^{\cR}_n\}}(u; C) = \lim_{n \to \infty} \f_{\e^{\mathcal{R}}_n}(\uCR, C)\,.
\end{equation}
We claim that
\begin{itemize}
\item[(C1)] the set function $\lambda:\cA(\o)\to[0,\infty)$ given by
\[
\lambda(B):=\W_{\{\e^{\cR}_n\}}(u; B)
\]
is a positive finite Radon measure absolutely continuous with respect to $\mu:=\hno \restr \partial^* A$, 
\item[(C2)] for $\hno$-a.e. $x_0\in\partial A$, it holds
\begin{equation}\label{eq:ubDiff}
\frac{d\lambda }{d \mu}(x_0) \leq \sigma(\nu(x_0))\,.
\end{equation}
\end{itemize}
This allows us to conclude. Indeed, we have that
\begin{align*}
\lim_{n \to \infty} \f_{\e^{\mathcal{R}}_n}(u^{\o}_{\e^{\cR}_n}, \o) &= \W_{\{\e^{\cR}_n\}}(u; \o)\\
& =\int_{\partial A_0} \frac{d\lambda }{d \mu}(x) \dhno(x) \\
&\leq \int_{\partial A_0} \sigma(\nu(x))\dhno(x) \\
&= \f_{0}(u).
\end{align*}

\emph{Step 1.} We first prove claim (C1).

We use the coincidence criterion in Lemma \ref{lem:coincidence} to show that $\lambda(B)$ is the restriction of a positive finite measure to $\cA(\o)$.

We will first prove (i) in Lemma \ref{lem:coincidence}. Let $U,V, W \in \cA(\o)$ be such that $\overline{U} \subset \subset V \subset W$. For $\del > 0$, let $V^{\del}$ and $W^{\del}$ be two elements of $\cR$ such that $V^{\del} \subset V$, $W^{\del} \subset W \setminus \overline{U}$, and
\begin{equation}\label{eq:hnodel}
\hno\left(\, \partial^* A_0\cap (W \setminus (V^{\del} \cup W^{\del})) \,\right) < \del.
\end{equation}

\begin{figure}\label{fig:sets}
\includegraphics[scale=0.6]{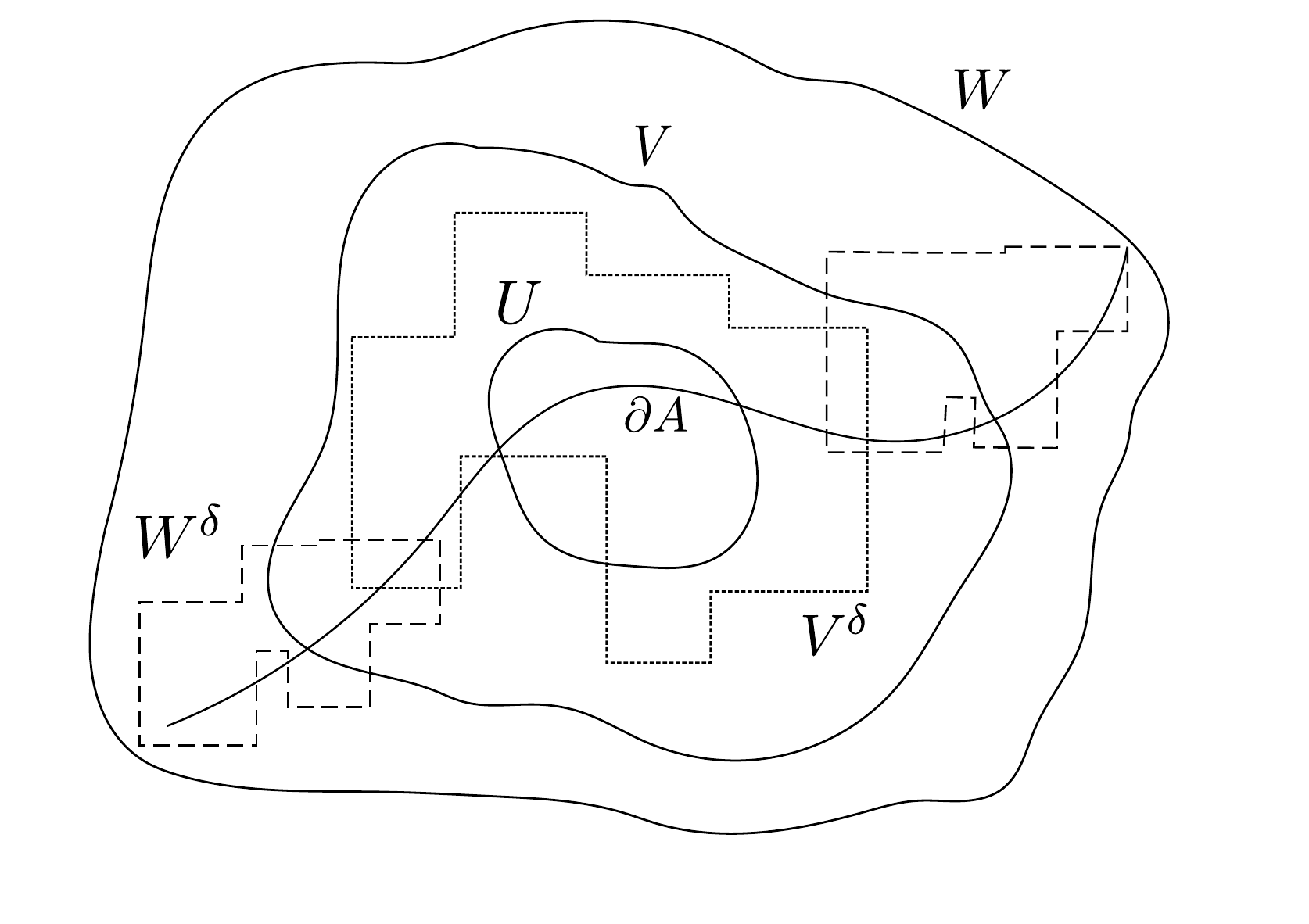}
\caption{The sets $\overline{U} \subset V \subset W$, and $V^{\del} \subset V$, $W^{\del} \subset W \setminus \overline{U}$.}
\end{figure}

Let $\{v_n\}_{n\in\N}\subset H^1(V^{\del}; \R^d)$ and $\{w_n\}_{n\in\N}\subset H^1(W^{\del}; \R^d)$ be such that $v_n\to u$ in $L^1(V^{\del}; \R^d)$, $w_n \to u$ in $L^1(W^{\del}; \R^d)$, and (see (\ref{eq:limUC}))
\begin{equation}\label{eq:witnessB}
\W_{\{\e^{\cR}_n\}}(u; V^{\del}) = \lim_{n \to \infty} \f_{\e^\cR_n}(v_n, V^\del)\,,
\end{equation}
\begin{equation}\label{eq:witnessD}
\W_{\{\e^{\cR}_n\}}(u; W^{\del}) = \lim_{n \to \infty} \f_{\e^\cR_n}(w_n, W^\del)\,.
\end{equation}

Let $\rho : \R^N \to [0, + \infty)$ be a symmetric mollifier, and define
\begin{equation}\label{eq:moll}
\xi_n(x) := \frac{1}{(\e_n^{\cR})^N} \rho \Big( \frac{x}{\e_n^{\cR}} \Big)\,.
\end{equation}
From Remark \ref{rem:DG} we can assume that $w_n = \xi_n * u $ on $\partial W^{\del}$
and $v_n=\xi_n * u $ on $\partial V^{\del}$.
Using a similar argument to the one found in Lemma \ref{lem:DeGslicing} applied to the sets $E_n:=( W^\del\setminus V^\del ) \setminus (W^\del\setminus V^\del)^{(n)}$,
and $E:=W^\del\setminus V^\del$ with boundary data $\xi_n\ast u$,
it is possible to find functions $\{\varphi_n\}\subset C^\infty(W^\delta)$ with
$\supp \nabla \varphi_n\subset L^{(i_0)}_n$ (here we are using the notation of the proof of Lemma \ref{lem:DeGslicing}) such that, if we define the function $u_n:W\to\R^d$ as
\[
u_n := \chi_{V^{\delta} \cup W^{\delta}}\left(\varphi_n v_n + (1- \varphi_n) w_n \right)
    + \chi_{(W \setminus(V^{\delta} \cup W^{\delta})}(\xi_n * u )\,,
\]
we have that $u_n\in H^1(W;\R^d)$ and
\begin{equation}\label{eq:enlayerto0}
\lim_{n\to\infty}\f_{\e^{\cR}_n}(u_n, L^{(i_0)}_n)=0\,.
\end{equation}
Notice that $u_n \to u$ in $L^1(W; \R^d)$ as $n \to \infty$.
Moreover, we get
\begin{align}\label{eq:layerBD}
\notag
\W_{\{\e^{\cR}_n\}}(u; W) &\leq \liminf_{n \to \infty}  \f_{\e^{\cR}_n}(u_n, W)\\ \notag
&\leq \liminf_{n \to \infty} \Big[\, \f_{\e^{\cR}_n}(u_n, V^\del) + \f_{\e^{\cR}_n}(u_n, W^\del)\\ \notag
&\hspace{2cm}+ \f_{\e^{\cR}_n}(u_n, W\setminus (V^\del\cup W^\del)) + \f_{\e^{\cR}_n}(u_n, L^{(i_0)}_n) \,\Big] \\\notag
&\leq \W_{\{\e^{\cR}_n\}}(u; V^{\del})+\W_{\{\e^{\cR}_n\}}(u; W^{\del})\\ 
&\hspace{2cm}+\liminf_{n \to \infty} \f_{\e^{\cR}_n}(u_n, W\setminus (V^\del\cup W^\del))
\end{align}
where in the last step we used \eqref{eq:witnessB}, \eqref{eq:witnessD} and \eqref{eq:enlayerto0}.
We see that
\begin{align}\label{eq:estperime}
\notag
\liminf_{n \to \infty}&\f_{\e^{\cR}_n}(u_n, W\setminus (V^\del\cup W^\del))\\ \notag
&=\liminf_{n \to \infty}\f_{\e^{\cR}_n}\left(\xi_n\ast u, \{x \in W\setminus (V^\del\cup W^\del) \,:\,
    \textrm{dist}(x, \partial A) \leq \e^{\cR}_n \}\right)\\ \notag
& \leq C \liminf_{n \to \infty} \frac{\l^N(\{ x \in W\setminus (V^\del\cup W^\del) \,:\,     
    \textrm{dist}(x, \partial A) \leq \e^{\cR}_n \})}{\e^{\cR}_n}\\ \notag
& = C \hno\left(\partial A\cap(W\setminus (V^\del\cup W^\del))\right)\\
&\leq C\delta\,,
\end{align}
where in the last step we used \eqref{eq:hnodel}.
Using \eqref{eq:layerBD}, \eqref{eq:estperime} and the fact that $V^\del\subset V$ and $W^\del \subset W\setminus\overline{U}$, we get
\begin{align*}
\W_{\{\e^{\cR}_n\}}(u;W) &\leq C\del + \W_{\{\e^{\cR}_n\}}(u; V^{\del})
    + \W_{\{\e^{\cR}_n\}}(u; W^{\del})\\
&\leq C\del + \W_{\{\e^{\cR}_n\}}(u;V) + \W_{\{\e^{\cR}_n\}}(u; W \setminus \overline{U})\,.
\end{align*}
Letting $\del \to 0^+$, we obtain (i).\\

We proceed to proving (ii) in Lemma \ref{lem:coincidence}. Let $U,V \in \cA(\o)$ be such that $U \cap V = \emptyset$. Fixing $\eta >0$, we can find $u_n \in H^1(U \cup V; \R^d)$ such that $u_n \to u$ and
\[
\W_{\{\e^{\cR}_n\}}(u; U \cup V) \geq \liminf_{n \to \infty} \f_{\e^\cR_n}(u_n, U \cup V) - \eta.  \] 

Then, since the restriction of $u_n$ to $U$ and $V$ converges to $u$ in these sets,
\[
\W_{\{\e^{\cR}_n\}}(u; U) \leq \liminf_{n \to \infty} \f_{\e^\cR_n}(u_n, U)
\]
and
\[\W_{\{\e^{\cR}_n\}}(u; V) \leq \liminf_{n \to \infty} \f_{\e^\cR_n}(u_n, V)\]
by definition, we have
\begin{align*}
\lambda(U) + \lambda(V) &\leq \liminf_{n \to \infty} \f_{\e^\cR_n}(u_n, U) + \liminf_{n \to \infty} \f_{\e^\cR_n}(u_n, V) \\
&\leq \liminf_{n \to \infty}  \f_{\e^\cR_n}(u_n, U \cup V) \leq \lambda(U \cup V) + \eta.
\end{align*}
Sending $\eta \to 0^+$, we conclude
\[ 
\lambda(U) + \lambda(V) \leq \lambda(U \cap V).
\]
To prove the opposite inequality, as in the proof of (i), we select $U^{\delta} \subset U$, $V^{\delta} \subset V$ with $U^{\delta}, V^{\delta} \in \cR$ and
\begin{equation}\label{eq:boundarysmall}
\hno \left( \partial^*A_0 \cap \left(\left(U \cup V \right) \setminus \left( U^\delta \cup V^\del \right) \right) \right) < \delta.
\end{equation}
Again we may select $v_n \in H^1(V^{\delta};\R^d)$ and $u_n \in H^1(U^{\delta}; \R^d)$ such that $v_n \to u$ in $L^1(V^{\delta}; \R^d)$, $u_n \to u$ in $L^1(U^{\delta}; \R^d)$ and
\begin{equation}\label{eq:limU}
\W_{\{\e^{\cR}_n\}}(u_n; U^{\delta}) \leq \liminf_{n \to \infty} \f_{\e^\cR_n}(u_n, U^{\delta}),
\end{equation}
\begin{equation}\label{eq:limV}
\W_{\{\e^{\cR}_n\}}(v_n; V^{\delta}) \leq \liminf_{n \to \infty} \f_{\e^\cR_n}(v_n, V^{\delta}).
\end{equation}
As in the proof of (i) of Lemma \ref{lem:coincidence}, we may assume without loss of generality that $u_n = \xi_n * u$ on $\partial U^{\delta}$, $v_n = \xi_n * v$ on $\partial V^{\delta}$, and we can find functions $\varphi_n \in C^{\infty}(U \cap V; [0,1])$ so that, defining
\[
w_n := \chi_{U^\delta\cup V^\delta}(\varphi_n u_n + (1-\varphi_n) v_n) + \chi_{(U \cup V) \setminus (U^{\delta} \cup V^{\delta})} \xi_n * u
\]
we have $w_n\in H^1(U\cup V;\R^d)$ and
\begin{equation}\label{eq:thinstrip}
\lim_{n \to \infty}  \f_{\e^\cR_n}(w_n, L^{(i_0)}_n) = 0,
\end{equation}
where $\nabla \varphi_n \subset L^{(i_0)}_n$, again using the notation of Lemma \ref{lem:DeGslicing}. Observing that $w_n \to u$ in $L^1(U \cup V; \R^d)$, we get
\begin{align*}
\lambda(U \cup V) &\leq \liminf_{n \to \infty} \f_{\e^\cR_n}(w_n, U \cup V) \\
&\leq \bigg[  \f_{\e^\cR_n}(u_n, U^{\delta}) + \f_{\e^\cR_n}(v_n, V^{\delta}) \\
& \hspace{1cm}+ \f_{\e^\cR_n}(w_n, (U \cup V) \setminus (U^{\delta} \cup V^{\delta})) + \f_{\e^\cR_n}(w_n, L^{(i_0}_n) \bigg] \\
&\leq \lambda(U^{\delta}) + \lambda(V^{\delta}) + \liminf_{n \to \infty} \f_{\e^\cR_n}(w_n, (U \cup V) \setminus (U^{\delta} \cup V^{\delta}))
\end{align*}
where in the last step we used \eqref{eq:limU}, \eqref{eq:limV}, and \eqref{eq:thinstrip}. Noticing that
%as in (ii) that
\[ \liminf_{n \to \infty} \f_{\e^\cR_n}(w_n, (U \cup V) \setminus (U^{\delta} \cup V^{\delta})) \leq C \hno \left( \partial^*A_0 \cap \left(\left(U \cup V \right) \setminus \left( U^\delta \cup V^\del \right) \right) \right) \]
and by \eqref{eq:boundarysmall} we have
\[
\lambda(U \cup V) \leq \lambda(U^{\delta}) + \lambda(V^{\delta}) + C \delta 
\leq \lambda(U) + \lambda(V) + C \delta
\]
and, letting $\delta \to 0$, we conclude (ii).\\

We prove (iii) in Lemma \ref{lem:coincidence}. Let $\o' \subset \subset \o$. Recalling \eqref{eq:moll}, we know that
$u * \xi_n$ is constant outside a tubular neighborhood of width $\e^{\mathcal{R}}_n$ around $\partial^* A$ and that $\|\nabla(u * \xi_n)\|_{L^\infty}\leq \frac{C}{\e^{\mathcal{R}}_n}$. Thus
\begin{align}\label{eq:estlambda}
\lambda(\o')=\W_{\{\e^{\cR}_n\}}(u; \o') \leq \liminf_{n \to \infty} \f_{\e^{\mathcal{R}}_n}( u * \xi_n, \o')
    \leq C  \hno(\o' \cap \partial^* A) = C \mu(\o'). 
\end{align}
This shows, by the coincidence criterion Lemma \ref{lem:coincidence}, that $\lambda \restr \o'$ is a Radon measure.
Since $\mu$ is a finite Radon measure in $\o$ and \eqref{eq:estlambda} holds for every $\o' \subset \subset \o$,
we conclude that $\lambda$ is a finite Radon measure in $\o$ absolutely continuous with respect to $\mu$, which was the claim (C1).\\

\emph{Step 2.} We now prove (C2).
Let $x_0 \in \o \cap \partial^* A$ be on a face of $\partial^* A$ (since the set is polyhedral) and write $\nu:=\nu_A(x_0)$.
Using Proposition \ref{prop:periodW} it is possible to find a rotation $R_\nu$ and $T\in\N$ such that,
setting $Q_\nu:=R_\nu Q$, we get $Q_\nu\in\mathcal{Q}_\nu$ and
\[
W(x+nTv,p)=W(x,p)\,,
\]
for a.e. $x\in\o$, every $v\in\Lambda$ that is orthogonal to one face of $Q_\nu$, every $p\in\R^N$ and $n \in \N$.
By Remark \ref{rem:newdefnormal} it follows that for $\mu$-almost every $x_0 \in \o$,
\begin{equation}\label{eq:diffEps}
\frac{d\lambda }{d \mu }(x_0) = \lim_{\e \to 0^+} \frac{\lambda( Q_{\nu}(x_0, \e)) }{\e^{N-1}},
\end{equation}
where $Q_{\nu}(x_0, \e) := x_0 + \e Q$.
In view of Lemma \ref{lem:sigmaQ}, it is possible to find
$\{T_k\}_{k\in\N} \subset T \N$ with $T_k \to \infty$ as $k\to\infty$, and
$\{u_k\}_{k\in\N} \subset \mathcal{C}(Q_{\nu}, T_k)$ such that
\begin{align}\label{eq:sigwitness}
\sigma^{Q_\nu}(\nu) &= \lim_{k \to \infty} \frac{1}{T_k^{N-1}} \int_{T_kQ_{\nu}} \bigg[ W(y, u_k(y)) + |\nabla u_k(y)|^2 \bigg] dy \nonumber\\
&=\lim_{k \to \infty} \int_{Q_\nu} \bigg[ T_k W(T_k x, v_k(x)) + \frac{1}{T_k} |\nabla v_k(x)|^2 \bigg] dx,
\end{align}
where $v_k:Q_\nu\to\R^d$ is defined as $v_k(x) := u_k(T_kx)$ and $\sigma^{Q_\nu}(\nu)$ is defined as in Lemma \ref{lem:sigmaQ}.
Without loss of generality, by density, we can assume $u_k\in \mathcal{C}(Q_{\nu}, T_k)\cap L^\infty(T_k Q_\nu;\R^d)$.
Since the choice of mollifier $\rho\in C^\infty_c(B(0,1))$ is arbitrary by Lemma \ref{lem:indepmoll}, we will assume here that $\supp \rho \subset B(0,\frac{1}{2})$ and thus
\[
u_k \left( T_k x \right) = u_{0, \nu}(x)\quad\textrm{ if }\,\,|T_k x| \geq \frac{1}{2}\,.
\] 
For $x\in\R^N$ let $x_\nu:=x\cdot\nu$ and $x':=x-x_\nu \nu$. Moreover, set $Q'_\nu:=Q_\nu\cap\nu^\perp$.

For $t \in \left(-\frac{1}{2}, \frac{1}{2}\right)$, extend the function $x'\mapsto v_k(x' + t\nu)$ to the whole $\nu^\perp$ by periodicity, and define
\begin{equation}\label{def:vnk}
v^{(\e)}_{n,k}(x) :=
\begin{cases} 
      u_{0, \nu}(x) & \textrm{ if } |x_{\nu}| > \frac{\e^{\cR}_n T_k}{2 \e}, \\
      &\\
      v_k\Big(\frac{\e x}{\e^{\cR}_n T_k}\Big) & \textrm{ if } |x_{\nu}| \leq \frac{\e^{\cR}_n T_k}{2 \e}.
\end{cases}
\end{equation}
The idea behind the definition of the function $v^{(\e)}_{n,k}$ is the following (see Figure \ref{fig:recovery}): for every fixed $\e>0$ and $k\in\N$ we are tiling the face of $A$ orthogonal to $\nu$ with $\e^{\cR}_n$-rescaled copies of the optimal profile $u_k$.
The fact that $A$ is a $\Lambda$-polyhedral set and that $T_k\in T\N$ ensure that it is possible to use the periodicity of $W$ to estimate the energy in each cube of edge length $\frac{\e^{\cR}_n T_k}{\e}$.
The presence of the factor $\e$ in \eqref{def:vnk} localizes the function around the point $x_0$ and accommodates the blow-up method we are using to prove the limsup inequality and, because of periodicity, will play no essential role in the fundamental estimate \eqref{eq:limit}.

\begin{figure}
\includegraphics[scale=0.6]{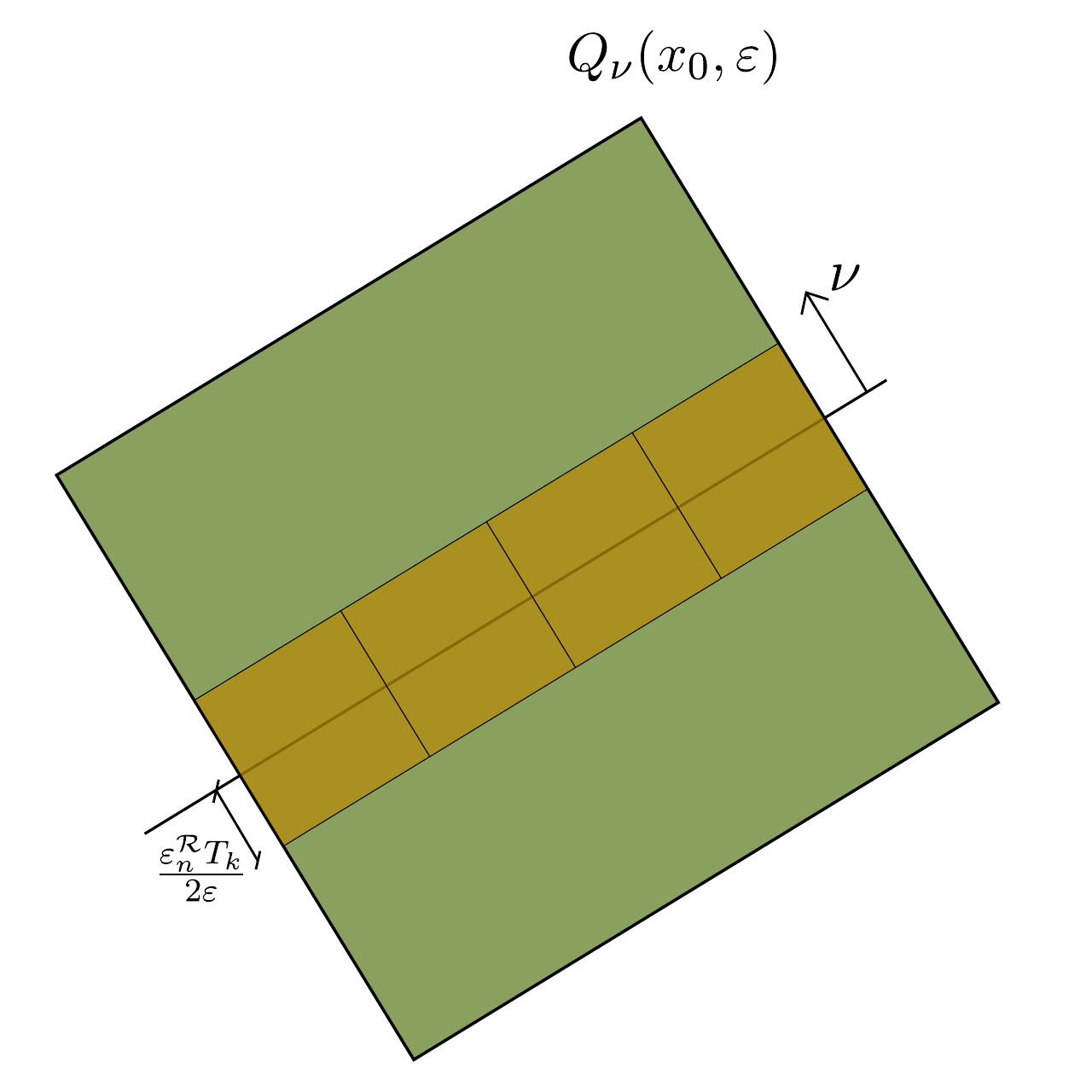}
\caption{The construction of the recovery sequence $v^{(\e)}_{n,k}$: for every $\e>0$ and $k\in\N$ fixed, we defined it as $u_{0, \nu}$ in the green region and, in each yellow square of side length $\frac{\e^{\cR}_n T_k}{\e}$, as a rescaled version of the function $u_k$.}
\label{fig:recovery}
\end{figure}

Let $m_n \in R_\nu\left(T \Z^N\right)$ and $s_n \in [0,T)^N$ be such that $\frac{x_0}{\e^{\cR}_n} = m_n + s_n$, and let
\[
x_{\e, n} := - \frac{\e^{\cR}_n}{\e} s_n\,.
\]
Note that for every $\e > 0$ we have 
\begin{equation}\label{eq:centers}
\lim_{n \to \infty} x_{\e, n} = 0. 
\end{equation}
Define the functions $u_{n, \e, k} \in H^1(Q_{\nu}(x_0, \e); \R^d)$ by
\[
u_{n,\e,k}(x) := v^{(\e)}_{n,k}\left( \frac{x-x_0}{\e} - x_{\e, n} \right).
\]
We claim that there is $\e'(x_0)$ such that for every $0< \e < \e'(x_0)$ and any $k\in\N$
\begin{equation}\label{eq:unekL1}
\lim_{n \to \infty} \|u_{n, \e, k} - u \|_{L^1(Q_{\nu}(x_0, \e); \R^d)} = 0.
\end{equation}
Since $x_0$ is on a face of $\partial^* A$, we can find $\e'$ such that $u = u_{0, \nu}(x-x_0)$ in $Q_{\nu}(x_0, \e')$. Changing variables,
\begin{align*}
&\int_{Q_{\nu}(x_0, \e)} |u_{n, \e, k}(x) - u(x)| \dd x \\
&\hspace{2cm}= \int_{(\e Q_\nu - \e x_{\e,n}) \cap \{z : |z_\nu| \leq \frac{\e^{\cR}_n T_k}{2}  \} }
    \bigg| v^{(\e)}_{n,k} \bigg( \frac{z}{\e} \bigg) - u(x_0 +  z + \e x_{\e,n}) \bigg| \dd z. \\
\end{align*}
Since the functions $v^{(\e)}_{n,k}$ are uniformly bounded with respect to $n\in\N$, we prove our claim by noticing that $|(\e Q_\nu - \e x_{\e,n}) \cap \{z : |z_\nu| \leq \frac{\e^{\cR}_n T_k}{2}  \}|\to0$ as $n\to\infty$.

Thus, using the definition of $\lambda$ and \eqref{eq:unekL1}, we get
\begin{equation}\label{eq:estdensabove}
\frac{\lambda(Q_{\nu}(x_0, \e)) }{\e^{N-1}} \leq \liminf_{n \to \infty} \frac{1}{\e^{N-1}} \f_{\e^{\cR}_n}(u_{n,\e,k}, Q_\nu(x_0,\e)). 
\end{equation}
We want to rewrite the right-hand side of \eqref{eq:estdensabove} in terms of the functions $v^{\e}_{n,k}$.
To do so, changing variables, we write
\begin{align*}
&\frac{1}{\e^{N-1}} \f_{\e^{\cR}_n}(u_{n,\e,k}, Q_\nu(x_0,\e))\\
&\hspace{0.8cm}=\int_{Q_\nu} \left[\, \frac{\e}{\e^{\cR}_n} W \bigg( \frac{x_0 + \e y}{\e^{\cR}_n}, u_{n, \e, k}(x_0+ \e y) \bigg)
    + \e \e^{\cR}_n | \nabla u_{n, \e, k}(x_0 + \e y)|^2 \,\right] \dd y \\
&\hspace{0.8cm} = \int_{Q_\nu} \left[\, \frac{\e}{\e^{\cR}_n} W \bigg( \frac{x_0 +\e y}{\e^{\cR}_n},
    v^{(\e)}_{n, k}(y - x_{\e,n}) \bigg)+ \frac{\e^{\cR}_n}{\e} | \nabla  v^{(\e)}_{n, k}(y - x_{\e,n})|^2 \,\right] \dd y\\
&\hspace{1cm}=  \int_{Q_\nu-x_{\e, n}} \left[\, \frac{\e}{\e^{\cR}_n} W \bigg( \frac{x_0 + \e (y+x_{\e,n})}{\e^{\cR}_n},
    v^{(\e)}_{n, k}(y) \bigg) + \frac{\e^{\cR}_n}{\e} | \nabla  v^{(\e)}_{n, k}(y)|^2 \,\right] \dd y \\
&\hspace{0.8cm}=  \int_{Q_\nu-x_{\e, n}} \left[\, \frac{\e}{\e^{\cR}_n} W \bigg( m_n + \frac{\e y}{\e^{\cR}_n},
    v^{(\e)}_{n, k}(y) \bigg) + \frac{\e^{\cR}_n}{\e} | \nabla  v^{(\e)}_{n, k}(y)|^2 \,\right] \dd y \\
&\hspace{0.8cm}=  \int_{Q_\nu-x_{\e, n}} \left[\, \frac{\e}{\e^{\cR}_n} W \bigg(\frac{\e y}{\e^{\cR}_n}, v^{(\e)}_{n, k}(y) \bigg)
    + \frac{\e^{\cR}_n}{\e} | \nabla  v^{(\e)}_{n, k}(y)|^2 \,\right] \dd y\\
&\hspace{0.8cm}=\f_{\frac{\e^{\cR}_n}{\e}}(v^{(\e)}_{n,k}, Q_\nu-x_{\e, n})\,,
\end{align*}
where in the second to last step we used the periodicity of $W$.

We claim that
\begin{equation}\label{eq:shiftBound}
\limsup_{k \to \infty} \limsup_{\e \to 0^+} \limsup_{n \to \infty}
    \f_{\frac{\e^{\cR}_n}{\e}}\left(v^{\e}_{n,k}, (Q_\nu-x_{\e,n})\setminus Q_\nu\right)  = 0.
\end{equation}
Indeed, using Fubini's Theorem and a change of variables, we have
\begin{align*}
&\f_{\frac{\e^{\cR}_n}{\e}}\left(v^{\e}_{n,k}, (Q_\nu-x_{\e,n})\setminus Q_\nu\right) \\
&\hspace{1cm} = \int_{-\frac{\e^{\cR}_n T_k}{2 \e}}^{\frac{\e^{\cR}_n T_k}{2 \e}} \int_{(Q'_\nu-x_{\e, n})\setminus Q'_\nu} \left[\, \frac{\e}{\e^{\cR}_n} W \bigg(\frac{\e y}{\e^{\cR}_n}, v^{(\e)}_{n, k}(y) \bigg) + \frac{\e^{\cR}_n}{\e} | \nabla  v^{(\e)}_{n, k}(y)|^2 \,\right] \dd y \\
&\hspace{1cm} = \int_{-\frac{1}{2}}^{\frac{1}{2}} \int_{(Q'_\nu-x_{\e, n})\setminus Q'_\nu} \Bigg[\, T_k W \bigg( \bigg( T_k\frac{\e x'}{\e^{\cR}_n T_k} + T_k  x_\nu \nu  \bigg), v_k \bigg( \frac{\e x'}{\e^{\cR}_n T_k} + x_\nu \nu \bigg) \bigg) \\
& \hspace{7cm}+ \frac{1}{T_k} \bigg| \nabla v_k \bigg( \frac{\e x'}{\e^{\cR}_n T_k} + x_\nu \nu  \bigg) \bigg|^2 \,\Bigg]\dd \mathcal{H}^{N-1} (x') \dd x_{\nu}\,.
\end{align*}
Fix $k\in\N$. By \eqref{eq:centers}, for each $\e > 0$, let $n( \e) \in \N$ be such that $|x_{\e, n}| < \e$ for all $n \geq n(\e)$.
In particular, we have $(Q'_\nu-x_{\e, n})\setminus Q'_\nu \subset (1+\e)Q'_\nu \setminus Q'_\nu$.
Set $\mu^{\e,k}_n:=\frac{\e}{\e^{\cR}_n T_k}$. For every $x_\nu\in(-\frac{1}{2},\frac{1}{2})$, the functions
$f,g:Q'_\nu\to\R$ defined by
\[
f(x'):= W\left( (T_k x' + T_k x_\nu \nu ), v_k(x' + x_\nu \nu  ) \right),\quad
g(x'):= \bigg| \nabla v_k \bigg(  x' +\nu  x_\nu  \bigg) \bigg|^2
\]
are $Q'_{\nu}$ periodic. The Riemann-Lebesgue Lemma yields
\begin{align}\label{eq:RLf}
&\lim_{n\to\infty} \int_{U} f(\mu^{\e,k}_n x') \dd \mathcal{H}^{N-1}( x') \nonumber \\
&\hspace{3cm}= |U|\int_{Q'_\nu} W\left( (T_k x' + T_k x_\nu \nu), v_k(x' + x_\nu \nu) \right) \dd \mathcal{H}^{N-1}(x')
\end{align}
and
\begin{equation}\label{eq:RLg}
\lim_{n\to\infty} \int_{U} g( \mu^{\e,k}_n  x') \dd \mathcal{H}^{N-1}( x') = |U|\int_{Q'_\nu} \left| \nabla v_k \left( x' + x_\nu \nu \right) \right|^2 \dd \mathcal{H}^{N-1}( x') \,,
\end{equation}
for every open and bounded set $U\subset\R^{N-1}$. Thus we get
\begin{align*}
&\limsup_{n \to \infty}  \f_{\frac{\e^{\cR}_n}{\e}}\left(v^{\e}_{n,k}, (Q_\nu-x_{\e,n})\setminus Q_\nu\right)\\
& \hspace{1cm} \leq \limsup_{n \to \infty} \int_{-\frac{1}{2}}^{\frac{1}{2}} \int_{(1+ \e) Q'_{\nu} \setminus Q'_{\nu}} \Bigg[\, T_k W \bigg( \bigg( T_k\frac{\e x'}{\e^{\cR}_n T_k} + T_k  x_\nu \nu \bigg), v_k \bigg( \frac{\e x'}{\e^{\cR}_n T_k} + x_\nu \nu \bigg) \bigg) \\
& \hspace{7cm}+ \frac{1}{T_k} \bigg| \nabla v_k \bigg( \frac{\e x'}{\e^{\cR}_n T_k} + x_\nu \nu \bigg) \bigg|^2 \,\Bigg]\dd \mathcal{H}^{N-1} (x') \dd x_{\nu} \\
& \hspace{1cm} \leq \left|(1+\e)Q'_{\nu} \setminus Q'_{\nu}\right|  \left( \int_{Q_\nu} \left[ T_k W \left(T_k x, v_k(x) \right) + \frac{1}{T_k} \left| \nabla v_k(x)\right|^2 \right] dx \right).
\end{align*}
Sending $\e\to0$ we obtain \eqref{eq:shiftBound}.

Finally, we claim that 
\begin{equation}\label{eq:limit}
\limsup_{k\to\infty}\limsup_{\e\to0^+}\limsup_{n\to\infty} \f_{\frac{\e^{\cR}_n}{\e}}\left(v^{\e}_{n,k}, Q_\nu \right)
    =\sigma^{Q_\nu}(\nu)\,.
\end{equation}
Recalling the definition of the functions $v_{n,k}^{(\e)}$ (see \eqref{def:vnk}) and using Fubini's Theorem we can write
\begin{align*}
\f_{\frac{\e^{\cR}_n}{\e}}\left(v^{\e}_{n,k}, Q_\nu \right)&=
    \f_{\frac{\e^{\cR}_n}{\e}}\left(v^{\e}_{n,k}, Q_\nu\cap\left\{ |x_\nu|\leq \frac{\e^{\cR}_n T_k}{2\e} \right\} \right) \\
&= \int_{-\frac{\e^{\cR}_n T_k}{2 \e}}^{\frac{\e^{\cR}_n T_k}{2 \e}} \int_{Q'_\nu} \bigg[ \frac{\e}{\e^{\cR}_n}
    W \bigg(  \frac{\e  x'}{\e^{\cR}_n} + \frac{\e  x_\nu \nu}{\e^{\cR}_n} , v_k \bigg( \frac{\e x'}{\e^{\cR}_n T_k} + \frac{\e x_\nu \nu}{\e^{\cR}_n T_k}\bigg) \bigg) \\
&\hspace{4cm} + \frac{\e}{\e^{\cR}_n T_k^2} \bigg|\nabla v_k\bigg( \frac{\e x'}{\e^{\cR}_n T_k} +
    \frac{\e x_\nu \nu}{\e^{\cR}_n T_k}\bigg)\bigg|^2 \bigg] \dd x' \dd x_\nu \\
&= \int_{-\frac{1}{2}}^{\frac{1}{2}} \int_{Q'_\nu} \bigg[ T_k W \bigg(  T_k\frac{\e x'}{\e^{\cR}_n T_k} + T_k y_\nu \nu ,v_k \bigg( \frac{\e x'}{\e^{\cR}_n T_k} + y_\nu \nu  \bigg) \bigg) \\
&\hspace{4cm}+ \frac{1}{T_k} \bigg|\nabla v_k\bigg( \frac{\e x'}{\e^{\cR}_n T_k} + y_\nu \nu \bigg) \bigg|^2 \bigg] \dd x' \dd y_\nu .
\end{align*}
Thus, using \eqref{eq:RLf} and \eqref{eq:RLg} (that are independent of $\e$), we obtain
\begin{align*}
\notag
\lim_{k \to \infty} \lim_{\e \to 0^+} \lim_{n \to \infty} \f_{\frac{\e^{\cR}_n}{\e}}\left(v^{\e}_{n,k}, Q_\nu \right)&= \lim_{k \to \infty} \int_{Q_\nu} \bigg( T_k W(T_k x, v_k(x)) + \frac{1}{T_k} |\nabla v_k(x)|^2 \bigg) \dd x\\
& = \sigma^{Q_\nu}(\nu).
\end{align*}

From \eqref{eq:estdensabove}, \eqref{eq:shiftBound} and \eqref{eq:limit} we get
\begin{align}\label{eq:lastestimate}
\lim_{\e\to0}\frac{\lambda(Q_\nu(x_0, \e)) }{\e^{N-1}}&\leq
    \limsup_{k\to\infty}\limsup_{\e\to0^+}\limsup_{n\to\infty}\frac{1}{\e^{N-1}} \f_{\e^{\cR}_n}(u_{n,\e,k}, Q_\nu(x_0,\e)) 
        \nonumber\\
&\leq \sigma^{Q_\nu}(\nu)\,.
\end{align}
In order to conclude, we use Lemma \ref{lem:sigmaQ} to find a sequence $\{Q_n\}_{n\in\N}\subset\mathcal{Q}^\Lambda_\nu$ such that $\sigma^{Q_n}(\nu)\to\sigma(\nu)$ as $n\to\infty$. Using \eqref{eq:lastestimate} we obtain for every $n\in\N$
\[
\frac{d \lambda}{d \mu}(x_0) = \lim_{\e\to0}\frac{\lambda(Q_n(x_0, \e)) }{\e^{N-1}}\leq \sigma^{Q_n}(\nu)
\]
and, letting $n \to \infty$ we have
\[\frac{d \lambda}{d \mu}(x_0) \leq \sigma(\nu). \]
Using the Urysohn property, we conclude that if the set $A:=\{u=a\}$ is $\Lambda$-polyhedral, then there exists a sequence $\{u_n\}_{n\in\N}\subset H^1(\o;\R^d)$ with $u_n\to u$ in $L^1(\o;\R^d)$ such that
\[
\limsup_{n\to\infty}\f_{\e_n}(u_n)\leq\f_0(u).
\]
\\

\emph{Case 2.} We now consider the general case of a function $u\in BV(\o;\{a,b\})$.
Using Lemma \ref{lem:densitysets} it is possible to find a sequence of functions $\{v_k\}_{k\in\N}\subset BV(\o;\{a,b\})$ with the following properties: the set $A_k:=\{v_k=a\}$ is a $\Lambda$-polyhedral set and, setting $A:=\{u=a\}$, we have
\[
\lim_{k\to\infty}\|\chi_{A_k}-\chi_A\|_{L^1(\o)}=0\,,\quad\quad\quad
\lim_{k\to\infty}|P(A_k;\o)-P(A;\o)|=0\,.
\]
From the result of Case 1, for every $k\in\N$ it is possible to find a sequence
$\{u^{k}_n\}_{n\in\N}\subset H^1(\o;\R^d)$ with $u^{k}_n\to v_k$ as $n\to\infty$, such that
\[
\limsup_{n\to\infty} \f_{\e_n}(u^{k}_n)\leq \f_0(v_k)\,.
\]
Choose an increasing sequence $\{n(k)\}_{k\in\N}$ such that, setting $u_k:=u^{k}_{n(k)}$,
\begin{equation}\label{eq:estlimsup}
\|u_k-u\|_{L^1}\leq\frac{1}{k}\,,\quad\quad\quad \f_{\e_n}(u^{k}_n)\leq \f_0(v_k)+\frac{1}{k}\,.
\end{equation}
Recalling that the function $\sigma$ is upper semi-continuous on $\S^{N-1}$ (see Proposition \ref{prop:sigma}), from Theorem \ref{thm:resh} and \eqref{eq:estlimsup} we get
\[
\limsup_{k\to\infty} \f_{\e_{n(k)}}(u^k)\leq \limsup_{k\to\infty}  \f_0(v_k) \leq \f_0(u)\,.
\]
This concludes the proof of the limsup inequality.
\end{proof}

\section{Continuity of $\sigma$}

To prove that the function $\nu\mapsto\sigma(\nu)$ is continuous, notice that Theorem \ref{thm:main} implies, in particular, that the functional $\f_0$ is lower semi-continuous with respect to the $L^1$ convergence. It then follows from \cite[Theorem 5.11]{AFP} that the function
$\sigma$, when extended $1$-homogeneously to the whole $\R^N$, is convex.
Since $\sigma(\nu)<\infty$ for every $\nu\in\S^{N-1}$ (see Lemma \ref{lem:estimatesigma}), we also deduce that $\sigma$ is continuous.

For the convenience of the reader, we recall here the argument used in \cite[Theorem 5.11]{AFP} to prove convexity.
Take $v_0, v_1, v_2\in\R^N$ such that $v_0=v_1+v_2$.
We claim that $\sigma(v_0)\leq \sigma(v_1)+\sigma(v_2)$. Using the $1$-homogeneity of $\sigma$, this is equivalent to convexity.
To prove the claim, let $E:=\{x\in\o \,:\, x\cdot \nu_0\leq\alpha \}$, where $\alpha\in\R$ is such that $\o\setminus E\neq\emptyset$ and $\o\cap E\neq\emptyset$.
Let $X\subset\R^N$ be the the two dimensional space generated by $\nu_1$ and $\nu_2$, consider the unit two dimensional square $Q'$ and a triangle $T$ with outer normals $-\frac{\nu_0}{|\nu_0|}, \frac{\nu_1}{|\nu_1|}$ and $\frac{\nu_2}{|\nu_2|}$, and such that
\[
1,\quad \frac{|\nu_1|}{|\nu_0|},\quad \frac{|\nu_2|}{|\nu_0|}
\]
are the lengths of the side of $T$ orthogonal to $\nu_0, \nu_1$, and $\nu_2$ respectively.
Let $Q\subset\R^N$ be the unit cube and  $\widetilde{Q}:=\left\{ x'\in\R^{N-2} : (0,0,x') \in Q  \right\}$.
Let $R:\R^N\to\R^N$ be a rotation such that
$R\left(  \left\{  (x_1,x_2,0,\dots,0)\in\R^N : (x_1,x_2)\in (-1/2, 1/2)^2 \right\} \right) = Q'$.
Let $z\in\R^N$ and $r>0$ be such that  $z+rQ\subset\o\setminus E$.
Then there exists $\lambda$ such that $\lambda T\subset r Q'$.
For $n\in\N$, let
\[
E_{n}:=E\cup\bigcup_{i=1}^n\left(z_i+\frac{1}{n}\left( \lambda T \times \widetilde{Q} \right) \right),
\]
where the $z_i$'s are such that the elements in the second union are pairwise disjoint and $z_i+\left( \lambda T \times \widetilde{Q} \right) \subset z+rQ$.
It can be shown that $\chi_{E_n}\to\chi_E$, so that by lower semi-continuity of $\f_0$ we obtain
\[
0\leq \liminf_{n\to\infty} \left[\, \f_0(\chi_{E_n}) - \f_0(\chi_E)    \,\right] 
= \frac{1}{|\nu_0|}[\sigma(\nu_1)+\sigma(\nu_2)-\sigma(\nu_0)].
\]
This proves the claim.

%%%%%%%%%%%%%%%%%%%%%%%%%%%%%%%%%%%%%%%%%%%%%%%%
%%%%%%%%%%%%%%%%%%%%%%%%%%%%%%%%%%%%%%%%%%%%%%%%
%%%%%%%%%%%%%%%%%%%%%%%%%%%%%%%%%%%%%%%%%%%%%%%%
%%%%%%%%%%%%%%%%%%%%%%%%%%%%%%%%%%%%%%%%%%%%%%%%

\subsection*{Acknowledgement}
We would like to thank the Center for Nonlinear Analysis at Carnegie Mellon University for its support
during the preparation of the manuscript.
Riccardo Cristoferi, Irene Fonseca and Adrian Hagerty were supported by the National Science Foundation under Grant No. DMS-1411646.

%%%%%%%%%%%%%%%%%%%%%%%%%%%%%%%%%%%%%%%%%%%%%%%%
%%%%%%%%%%%%%%%%%%%%%%%%%%%%%%%%%%%%%%%%%%%%%%%%
%%%%%%%%%%%%%%%%%%%%%%%%%%%%%%%%%%%%%%%%%%%%%%%%
%%%%%%%%%%%%%%%%%%%%%%%%%%%%%%%%%%%%%%%%%%%%%%%%

\bibliographystyle{siam}
\bibliography{Bibliography}

\end{document}